\documentclass[12pt]{article} 
\usepackage{amsfonts}
\usepackage{amstext}
\usepackage{graphicx}
\usepackage{amsmath}
\usepackage{amssymb}
\usepackage{amsthm}
\usepackage{enumerate}

\renewcommand{\Re}{\mathop{\rm Re}\nolimits}
\renewcommand{\Im}{\mathop{\rm Im}\nolimits}

\graphicspath{{Graphics/}}
\newtheorem{theorem}{Theorem}
\newtheorem{proposition}{Proposition}
\newtheorem{remark}{Remark}
\newtheorem{cor}{Corollary}

\begin{document}
	
\title{Curved  wedges in the long-time asymptotics  for the integrable  nonlocal nonlinear Schr\"odinger equation}

\author{Ya. Rybalko$^{\dag}$ and D. Shepelsky$^{\dag,\ddag}$\\
\small\em {}$^\dag$ B.Verkin Institute for Low Temperature Physics and Engineering\\
\small\em {} of the National Academy of Sciences of Ukraine\\
\small\em {}$^\ddag$ V.Karazin Kharkiv National University}
\date{}

\maketitle

\begin{abstract}
	We consider the Cauchy problem for 
	the integrable nonlocal nonlinear Schr\"odinger (NNLS) equation
	$
	iq_{t}(x,t)+q_{xx}(x,t)+2 q^{2}(x,t)\bar{q}(-x,t)=0, \, x\in\mathbb{R},\,t>0,
	$
	with a step-like boundary values:
	$q(x,t)\to 0$ as $x\to-\infty$ and  $q(x,t)\to A$ as $x\to\infty$ for all 
	$t\geq0$, where $A>0$ is a  constant. 
In a recent paper, we presented the 
	  long-time asymptotics of the solution $q(x,t)$ of this problem 
	along the rays $x/t=C\ne 0$, where $C$ is a constant.
	In the present paper, we extend  the asymptotics 
	into a region that is asymptotically 
	closer to the ray $x=0$ than any of these rays.
	We specify a one-parameter family of wedges in the $x,t$-plane, with 
	curved boundaries, characterized by qualitatively different asymptotic behavior 
	of $q(x,t)$, and present the main asymptotic terms for each wedge.
	Particularly, for wedges within $x<0$, we show that  the solution decays as 
	$t^{p}\sqrt{\ln t}$ with  $p<0$  depending on the wedge. For wedges within $x>0$,
	we show that the asymptotics has an oscillatory nature, with the phase functions
	specific for each wedge and depending on a slow variable parametrizing the wedges.
	\end{abstract}

		
\section{Introduction}
We consider the following ``step-like problem'' for the focusing integrable nonlocal nonlinear Schr\"odinger (NNLS) equation ($\bar{q}$ denotes the complex conjugate of $q$)
\begin{subequations}\label{1}
\begin{align}
\label{1-a}
& iq_{t}(x,t)+q_{xx}(x,t)+2q^{2}(x,t)\bar{q}(-x,t)=0, & & x\in\mathbb{R},\,t>0,  \\
\label{1-b}
& q(x,0)=q_0(x), & &  x\in\mathbb{R}, 
\end{align}
\end{subequations}
with the boundary conditions to be satisfied for all $t\geq0$:
\begin{subequations}\label{2.0}
\begin{align}
\label{2-a}
& q(x,t)=o(1), & x\rightarrow-\infty,\\
\label{2-b}
& q(x,t)=A+o(1), & x\rightarrow+\infty.
\end{align}
\end{subequations}
Particularly, the initial data $q_0(x)$ is assumed to satisfy these boundary conditions as well:
\[
q_0(x)\to  \begin{cases}
0, & x\to -\infty,\\
 A, & x\to +\infty.
\end{cases}
\]

The NNLS equation (\ref{1-a}) is a relatively new integrable model introduced by M. Ablowitz and Z. Musslimani \cite{AMP} as a reduction 
(with $r(x,t)=\bar{q}(-x,t))$ of 
the coupled Schr\"odinger equations (also known as Ablowitz-Kaup-Newell-Segur (AKNS) system \cite{AKNS})
\begin{subequations}
\label{cs}
\begin{align}
\label{cs-a}
iq_t+q_{xx}+2q^2r=0,\\
\label{cs-b}
-ir_t+r_{xx}+2r^2q=0.
\end{align}
\end{subequations}
 The $2\times2$ Lax pair equations for the 
NNLS equation are as follows:
\begin{subequations}
\begin{align}
\label{LP}
&\Phi_{x}+ik\sigma_{3}\Phi= U(x,t)\Phi,\\
&\Phi_{t}+2ik^{2}\sigma_{3}\Phi= V(x,t,k)\Phi,
\end{align}
\end{subequations}
where $\sigma_3=\left(\begin{smallmatrix} 1& 0\\ 0 & -1\end{smallmatrix}\right)$, 
\[
\begin{split}
U(x,t)= &
\begin{pmatrix}
0& q(x,t)\\
-\bar{q}(-x,t)& 0
\end{pmatrix}, \\
V(x,t,k)=& \begin{pmatrix}
iq(x,t)\bar{q}(-x,t) & 2kq(x,t)+iq_{x}(x,t)\\
-2k\bar{q}(-x,t)+i(\bar{q}(-x,t))_{x}& -iq(x,t)\bar{q}(-x,t)
\end{pmatrix}.
\end{split}
\]

Notice that  (\ref{1-a}) can be viewed as the Schr\"odinger equation 
$iq_t(x,t)+q_{xx}(x,t)+V(x,t)q(x,t)=0$ with a potential $V(x,t)=2\bar{q}(-x,t)q(x,t)$ that satisfies the \textit{PT} symmetric \cite{BB} condition $V(x,t)=\overline{V(-x,t)}$. Therefore,  the NNLS equation is (i) integrable and (ii) a \textit{PT} symmetric system, which is of interest from both the physical and mathematical perspective.
Indeed, it is gauge equivalent to the unconventional system of coupled Landau-Lifshitz (CLL) equations in magnetics \cite{GA}. On the other hand, the NNLS equation is related to the quasi-monochromatic asymptotic reductions of the classical integrable equations: the Korteweg–de Vries (KdV) equation and Klein-Gordon equation \cite{AMJPA}.

Moreover, the NNLS equation is an abundant source of  new nonlinear effects, particularly, it supports both bright and dark soliton solutions simultaneously \cite{SMMC} (see also e.g. \cite{LXM}, \cite{XCLM} and references therein), which is in a sharp contrast with the conventional (focusing) nonlinear Schr\"odinger (NLS) equation ($r(x,t)=\bar{q}(x,t)$ in (\ref{cs})) that supports only bright soliton solutions. Also it has soliton solutions with both zero and nonzero backgrounds \cite{AMFL}, which, in general, blow up in a finite time. Additionally, (\ref{1-a}) supports rouge waves, algebro-geometric solutions and breather solutions (see \cite{YY, MS, GS, S} and references therein) with a number of interesting and distinctive properties. Furthermore, the NNLS equation is connected to the Alice-Bob systems \cite{Lou} describing various physical phenomena which can happen in two (or more)
places linked to each other. Finally, we notice that some unusual features are manifested in  the long-time asymptotics; particularly, the soliton-free solution of the initial value problem for (\ref{1-a}) with zero background  decays to zero with different  power rates along  different rays $\frac{x}{t}=const$ \cite{RS}. 

In \cite{RS2} we presented the long-time asymptotics of the solution of  problem (\ref{1}), (\ref{2.0}) along the rays $\xi\equiv\frac{x}{4t}=const\ne 0$. 
Particularly, we showed that in the case of the initial data that are close, in a certain sense,
 to the ``pure step data''
\begin{equation}
\label{step-ini}
q_{0A}(x) = \begin{cases} 0, & x<0, \\
A, & x>0,
\end{cases}
\end{equation}
the solution (asympotically) behaves qualitatively different in two
regions: $x<0$ (where the solution decay to $0$) and 
 $x>0$ (where it approaches a ``modulated'' constant).
Notice that in the case of the conventional focusing NLS equation in the similar
setting (with step-like initial data), 
the asymptotic picture is as follows:
(i) there are two sectors (in the $(x,t)$ half-plane), $\xi<C_1$ and $\xi>C_2$ with certain constants $C_1$ and $C_2$ 
with $C_1<C_2$,
 where the asymptotic behavior of the solution
is directly related to the (different) backgrounds in the initial data \cite{BKS};
(ii)  in the  \emph{transition sector}
$C_1<\xi<C_2$, the asymptotics is described in terms of elliptic functions with modulated
(by $\xi$) parameters.
Such
 behavior is intrinsic for other local integrable equations (KdV, defocusing NLS, Toda, etc.),
see, e.g., 
 \cite{EGKT, BiM, J, EMT}. With this respect, the NNLS equation
is qualitatively different: there is no room for a sector (with straight boundaries), in which the asymptotics for $x<0$ and $x>0$ could match.

Transition regions, being a rich source of  nonlinear effects,
have been a subject of interest since 1970s \cite{GP} for various problems for integrable systems. Particularly, in \cite{Kkdv} and \cite{KK}, the so-called
asymptotic solitons were  observed  for step-like problems for the KdV and the focusing NLS equations respectively (see also \cite{KMnls}). 
Notice that for problems with \emph{decaying boundary conditions}, 
transition regions with curved boundaries
(in the $x,t$ plane) can be specified as well. For instance, the Painlev\'e-II
transcendents describe the principal part of the asymptotics in a 
transition region for the mKdV equation \cite{DZ}
 (see also \cite{BM, CL} for more results on the transition zones for the mKdV equation). 
For the KdV equation, another transition zone 
(called the collisionless shock region \cite{AS}) was fully specified in 
\cite{DVZ}. Transition regions for other integrable equations have also been 
reported;  to name but a few, see, e.g.,  \cite{BIS, BS} for the Camassa--Holm equation and 
\cite{K} for the Toda lattice.

Integrable equations in the small dispersion (semi-classical) limit 
\cite{Du,BT,KMM,BuM}
are deeply related to the large time limit \cite{G16}.   
For the KdV equation, there is  a cusp-shaped (Whitham) region \cite{LL, Ven85},
characterized  by
 rapid modulated elliptic oscillations
whereas 
the description of the small dispersion behavior near its boundaries involves 
the Painlev\'e-I and Painlev\'e-II transcendents, see
\cite{DVZ1,  CG, GK} and references therein.

In the present paper we address the problem of extending the asymptotics obtained in \cite{RS2} into the zone connecting the sectors 
$\xi<-\varepsilon$ and $\xi>\varepsilon$ for any $\varepsilon>0$.
 We present a  family of curved transition zones parametrized by $\alpha\in(0,1)$ (see (\ref{scurve}) below), each characterized by qualitatively different 
behavior of the solution of problem (\ref{1}), (\ref{2.0}) and having its proper 
modulation parameter. Particularly, the decaying regimes for  $x<0$
have the rate of decay of order $t^{p(\alpha)}\sqrt{\ln t}$, where $p(\alpha)<0$ and $p(\alpha)$ can be written explicitly (see Theorem \ref{th1} below).
Moreover, the oscillating phases in these regimes involve 
a number of growing terms, of order 
 $t^{\frac{\alpha}{2-\alpha}}$, $\ln^2 t$, $\ln t\cdot\ln\ln t$, $\ln t$,
 and $\ln\ln t$.
For $x>0$, 
the asymptotics turns to be non-decaying and oscillatory,
 with oscillations 
characterized by  two large parameters, $\ln^2 t$ and $\ln t$. 
To the best of our knowledge, 
such picture of asymptotic regions for integrable systems has never been reported  before.

Our main tool is the adaptation of  the nonlinear steepest decent method (Deift and Zhou method), which was proposed by Deift and Zhou \cite{DZ} and relies on the previous works by Manakov \cite{M} and Its \cite{I1} (see \cite{DIZ} for the history of the problem). 
An important step of the nonlinear steepest decent method (as well as of the classical steepest decent method) consists in choosing  an appropriate 
``slow variable'' while keeping a fast variable in the phase factors
of the Riemann--Hilbert factorization problem 
associated with a particular problem for the nonlinear equation in question. 
If the slow variable $\xi=\frac{x}{4t}$ \cite{ZM} is appropriate, then 
it is natural to expect the  asymptotics to be  
qualitatively different in sectors (in the $x,t$-plane)  with straight boundaries.
On the other hand, the asymptotics in the collisionless shock region
 for the KdV equation \cite{DVZ} involves, as a slow variable,
 $s=-(12)^{3/2}\frac{\sqrt{t}\ln\frac{-x}{12t}}{(-x)^{3/2}}$  ($x<0$). 

In this paper, we introduce the slow variable $s=\frac{x^{2-\alpha}}{4t}$
in curved wedges parametrized by  $\alpha\in(0,1)$ 
(see Figure \ref{regions} below).  
This allows us to cope, to some extent, with the problem
of the asymptotic analysis of the Riemann--Hilbert problem
in the situation where the stationary phase point $k=-\xi$ merges with 
the singularity point $k = 0$ of  the associated spectral functions. 
In this way, we extend the region $\frac{x}{4t}=const\neq0$ where the large time asymptotics
of the solution of  problem (\ref{1}), (\ref{2.0}) was considered in \cite{RS2},
 to wedges with curved
boundaries: $\frac{x^2}{4s} >|t|> \frac{|x|}{4\xi}$ for any $s>0$ and $\xi>0$.

The paper is organized as follows. In Section \ref{IS} we briefly recall 
the Riemann-Hilbert problem formalism of 
the Inverse Scattering Transform (IST) method for problem (\ref{1}), (\ref{2.0}). The asymptotic analysis of the basic Riemann--Hilbert problem is presented in
 Section \ref{secas}, where the main result  (Theorem \ref{th1}) is proved. Finally, in Section \ref{merge} we show that the asymptotics in the sectors with 
 straight line boundaries obtained in \cite{RS2} are consistent with the limiting values 
(as $\alpha\to 1$) of the asymptotics in the curved sectors obtained in the present paper.

\section{Inverse scattering and the basic RH problem}\label{IS}
The IST method for the step-like problem (\ref{1}), (\ref{2.0}) based on the
Riemann-Hilbert problem formalism is presented   in \cite{RS2}.
In this section, we briefly recall the definition of the spectral functions, their main properties, and the formulation of the basic Riemann--Hilbert problem; the latter will be  a starting point of the asymptotic analysis presented in Section \ref{secas}. 

The $2\times2$ scattering matrix $S(k)$ ($\det S(k)=1$) is defined as follows:
\begin{equation}
\label{scmatr}
\Psi_1(x,t,k)=\Psi_2(x,t,k)e^{-(ikx+2ik^2t)\sigma_3}S(k)e^{(ikx+2ik^2t)\sigma_3},\quad k\in\mathbb{R}\setminus\{0\},
\end{equation}
where the $\Psi_j$, $j=1,2$ are related to the  Jost solutions of the Lax pair 
equations (\ref{LP}); they can be defined as 
the solutions of the  Volterra integral equations:
\begin{subequations}\label{psi}
\begin{align}
\label{psi-1}
&\Psi_1(x,t,k)=N_-(k)+\int^x_{-\infty}G_-(x,y,t,k)\left(U(y,t)-U_-\right)
\Psi_1(y,t,k)e^{ik(x-y)\sigma_3}\,dy,\\
\label{psi-2}
&\Psi_2(x,t,k)=N_+(k)+\int^x_{\infty}G_+(x,y,t,k)\left(U(y,t)-U_+\right)
\Psi_2(y,t,k)e^{ik(x-y)\sigma_3}\,dy,
\end{align}
\end{subequations}
where $N_+(k)=
\left(
\begin{smallmatrix}
1 & \frac{A}{2ik}\\
0 & 1
\end{smallmatrix}
\right)
$,  
$N_-(k)=
\left(
\begin{smallmatrix}
1 & 0\\
\frac{A}{2ik} & 1
\end{smallmatrix}
\right)
$,
$G_{\pm}(x,y,t,k)=\Phi_{\pm}(x,t,k)[\Phi_{\pm}(y,t,k)]^{-1}$ with
$\Phi_{\pm}(x,t,k)=N_{\pm}(k)e^{-(ikx+2ik^2t)\sigma_3}$,
$
U_+=
\left(
\begin{smallmatrix}
0 & A\\
0 & 0
\end{smallmatrix}
\right)
$
and
$
U_-=
\left(
\begin{smallmatrix}
0 & 0\\
-A & 0
\end{smallmatrix}
\right)
$.
From (\ref{scmatr}) and (\ref{psi}) 
evaluated for $t=0$ 
it follows that $S(k)=\Psi_1(0,0,k)[\Psi_2(0,0,k)]^{-1}$ is uniquely determined by the initial data $q_0(x)$ (involved in $U(x,0)$).

Taking into account the symmetry relation 
$\Lambda\overline{\Psi_1(-x,t,-\bar{k})}\Lambda^{-1}=\Psi_2(x,t,k)$, $k\in\mathbb{R}\setminus\{0\}$, where 
$\Lambda=\bigl(\begin{smallmatrix}0& 1\\1 & 0\end{smallmatrix}\bigl)$, the scattering matrix $S(k)$ can be written in  terms of three scalar functions 
$b(k)$, $a_1(k)$, and $a_2(k)$:
\begin{equation}\label{sm}
S(k)=
\begin{pmatrix}
a_1(k)& -\overline{b(-k)}\\
b(k)& a_2(k)
\end{pmatrix},\quad k\in\mathbb{R}\setminus\{0\}.
\end{equation}

\begin{proposition}\cite{RS2}\label{properties}
The spectral functions $a_j(k)$, j=1,2, and $b(k)$ have the following properties:
\begin{enumerate}
\item 
$a_{1}(k)$ is analytic in  $k\in\mathbb{C}^{+}$
and continuous in 
$\overline{\mathbb{C}^{+}}\setminus\{0\}$;
$a_{2}(k)$ is analytic in $k\in\mathbb{C}^{-}$
and continuous in $\overline{\mathbb{C}^{-}}$. 
Here $\overline{\mathbb{C}^{\pm}}=\{k\  |\  \pm\Im k\ge 0\}$.
\item
$a_{1}(k)=1+{O}\left(\frac{1}{k}\right)$
as $k\rightarrow\infty$ for  $k\in\overline{\mathbb{C}^+}$,
$a_{2}(k)=1+{O}\left(\frac{1}{k}\right)$
as $k\rightarrow\infty$ for  $k\in\overline{\mathbb{C}^-}$, and 
$b(k)={O}\left(\frac{1}{k}\right)$ as $k\rightarrow\infty$ for $k\in\mathbb{R}$.
\item
$\overline{a_{1}(-\bar{k})}=a_1(k)$,  
$k\in\overline{\mathbb{C}^{+}}\setminus\{0\}$; \qquad
$\overline{a_{2}(-\bar{k})}=a_2(k)$, $k\in\overline{\mathbb{C}^{-}}$.
\item
$a_{1}(k)a_{2}(k)+b(k)\overline{b(-
\bar{k})}=1$, $k\in{\mathbb R}\setminus\{0\}$.
\item 
$a_1(k)=\frac{A^2a_2(0)}{4k^2}+O(\frac{1}{k})$ as $k\to0$, $k\in \overline{\mathbb{C}^{+}}$ and $b(k)=\frac{Aa_2(0)}{2ik}+O(1)$ as $k\to0$, $k\in\mathbb{R}$.
\end{enumerate}
\end{proposition}

\begin{remark}
Notice that $a_2(0)$ can be calculated \cite{RS2} in terms of $q_0(x)$ as follows:
$a_2(0) = \frac{4(|v_2(0)|^2 - |v_1(0)|^2)}{A^2}$, where
$v_2(x)$ is the solution of the integral equation 
\[
v_2(x) = -\frac{iA}{2} -\int_{-\infty}^x
	\overline{q_0(-y)}\int_{-\infty}^y q_0(s) v_2(s)\,ds\,dy
\]
and $v_1(x)=\int_{-\infty}^x  q_0(s) v_2(s)\,ds$.
\end{remark}

In the case of pure step initial data $q_0(x)=q_{0A}(x)$, see (\ref{step-ini}), the scattering matrix (\ref{sm}) can be explicitly calculated by evaluating (\ref{scmatr})
and (\ref{psi}) at $x,t=0$:
\begin{equation}
\label{step}
S(k)=N_+^{-1}(k)N_-(k)=
\begin{pmatrix}
1+\frac{A^2}{4k^2}& -\frac{A}{2ik}\\
\frac{A}{2ik} & 1
\end{pmatrix}.
\end{equation}
In this  case, we have, particularly, that 
(a) $a_1(k)$ has one simple pure imaginary zero $k=i\frac{A}{2}$ in $\overline{\mathbb{C}^+}$ and 
(b) $a_2(k)$ has no zeros in $\overline{\mathbb{C}^-}$. Now observe that 
 these properties of the spectral functions survive 
any small $L^1$ perturbation of pure step initial data. This is clear for  $a_2(k)$;
as for $a_1(k)$, consider its values along the imaginary axis and notice that 
(i) in view of the symmetry property (see Item 3 of Proposition \ref{properties}), 
$a_1(k)$ is real-valued for $k=i\rho$, $0<\rho<\infty$,
(ii)  $a_1(i\rho)\to 1$ as $\rho\to\infty$ (see  Item 2 of Proposition \ref{properties}),
and 
(iii) $a_1(i\rho)\to -\infty$ as $\rho\to 0$ 
(see  Item 5 of Proposition \ref{properties}). 
Since in the case $q_0(x)=q_{0A}(x)$, $a_1(i\rho)$ is an increasing function 
with a single zero at $\rho=k_1=\frac{A}{2}$, it has a single simple zero on $0<\rho<\infty$ 
 for all $q_0(x)$ that are small perturbations of 
$q_{0A}(x)$  (see Figure \ref{za1}).

\begin{figure}[ht]
	\centering{\includegraphics[scale=0.4]{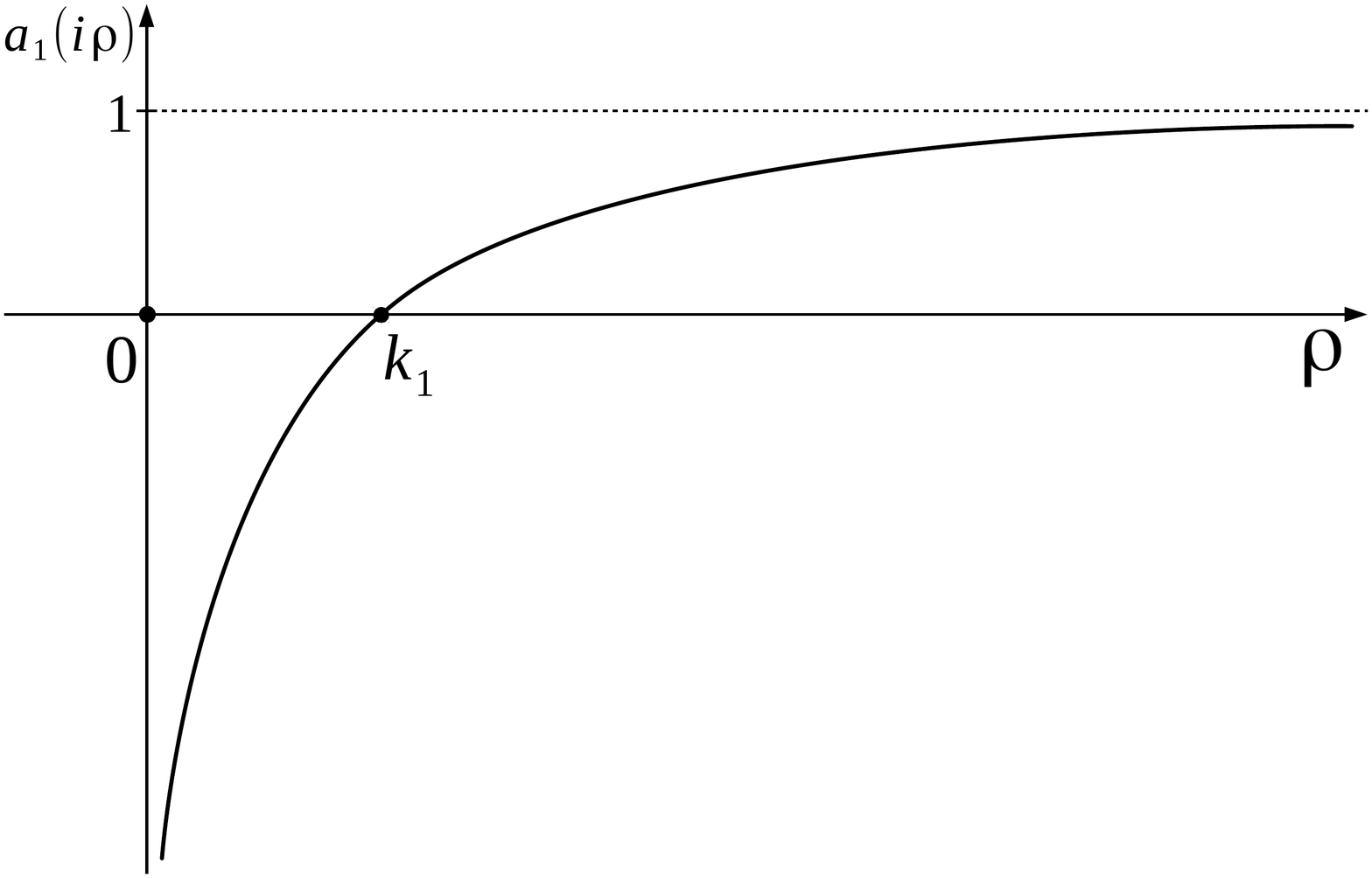}}
	\caption{Behavior of $a_1(k)$ on the imaginary axis for the initial data close to the pure step.}
	\label{za1}
\end{figure}

Another example of the problem with explicitly written spectral functions is
 the pure soliton solution
\begin{equation}
\label{pure-sol}
q_s(x,t)=\frac{A}{1-e^{-Ax-iA^2t+i\phi}},\quad \text{with some}\ \ \phi\in\mathbb{R},
\end{equation}
where the associated scattering matrix $S(k)$ is given by
\begin{equation}
\label{sp}
S(k)=
\begin{pmatrix}
\frac{k-i\frac{A}{2}}{k}& 0\\
0 & \frac{k}{k-i\frac{A}{2}}
\end{pmatrix}
\end{equation}
(for details, see Proposition 5 in \cite{RS2}). 
Particularly, in this case we have that 
$a_1(k)$ has one simple pure imaginary zero in $\overline{\mathbb{C}^+}$  at 
$k=i\frac{A}{2}$ and $a_2(k)$ has one simple zero in $\overline{\mathbb{C}^-}$
at $k=0$. In view of Item 5 of Proposition \ref{properties}, the latter 
property affects significantly the behavior of other spectral functions 
at $k=0$.

Motivated by the 
 discussion above (and similarly to \cite{RS2}), in what follows we consider 
two types of the Cauchy problem (\ref{1}), (\ref{2.0}) specified, in  spectral terms, by 


\begin{description}
\item [\underline{Assumption 1}]
\item [Case I:] (covering the case of pure step initial data and all its  small $L^1$ perturbations)\\
The spectral function $a_1(k)$ has one purely imaginary simple zero in $\overline{\mathbb{C}^+}$: $a_1(ik_1)=0$ with some $k_1>0$, 
and $a_2(k)$ has no zeros in $\overline{\mathbb{C}^-}$.
\item [Case II:] (covering the soliton solution)\\
The spectral function $a_1(k)$ has one purely imaginary simple zero in $\overline{\mathbb{C}^+}$: $a_1(ik_1)=0$ with some $k_1>0$,  and $a_2(k)$ has one simple zero in $\overline{\mathbb{C}^-}$ at $k=0$ (i.e., $a_2(0)=0$ and 
$a_{21}:=\lim\limits_{k\to0,\,k\in\overline{\mathbb{C}^-}}
\frac{a_2(k)}{k}\neq0$).
Additionally, we assume that  $a_{11}:=\lim\limits_{k\to 0,\,k\in\overline{\mathbb{C}^+}}ka_1(k)\not=0$.
\end{description}

\begin{proposition}\cite{RS2}\label{RH-representation}
Under Assumption 1, 
the solution of the initial problem (\ref{1}) with the boundary conditions (\ref{2.0}) can be represented either by
\begin{equation}\label{sol}
q(x,t)=2i\lim_{k\to\infty}kM_{12}(x,t,k),
\end{equation}
or, alternatively, by
\begin{equation}\label{sol1}
q(-x,t)=-2i\lim_{k\to\infty}k\overline{M_{21}(x,t,k)},
\end{equation}
where $M(x,t,k)$ is the solution of the Riemann--Hilbert problem,
with data determined in terms of the spectral functions associated
with the initial data $q_0(x)$:
find a $2\times2$-valued function $M(x,t,k)$, which is sectionally meromorphic with respect to $\mathbb{R}$ in the complex $k$-plane 
and satisfies the conditions:
\begin{subequations}\label{brh}
\begin{align}
\mathbf{(i)}\quad&M_+(x,t,k)=M_-(x,t,k)J(x,t,k) &\text{for a.e.}\  k\in\mathbb{R}\setminus\{0\},\\
\mathbf{(ii)}\quad&M(x,t,k)=I+O\left(\frac{1}{k}\right)&\mbox{uniformly as } k\to\infty,\\
\mathbf{(iii)}\quad&\underset{k=ik_1}{\operatorname{Res}} M^{(1)}(x,t,k)=
\frac{\gamma_1}{\dot{a}_1(ik_1)}e^{-2k_1x-4ik_1^2t}M^{(2)}(x,t,ik_1),&
\end{align}
\end{subequations}
where $M_\pm$ denotes the non-tangential limits of $M$ as $\Im k\to0$, $k\in\mathbb{C}^{\pm}$, here and below $M^{(j)}$, $j=1,2$ stands for the corresponding column of the matrix $M$ and
$$
J(x,t,k)=
\begin{pmatrix}
1+r_{1}(k)r_{2}(k)& r_{2}(k)e^{-2ikx-4ik^2t}\\
r_1(k)e^{2ikx+4ik^2t}& 1
\end{pmatrix},
$$
with $r_1(k)=\frac{b(k)}{a_1(k)}$ and $r_2(k)=\frac{\overline{b(-k)}}{a_2(k)}$,
$\gamma_1\in\mathbb C$ (with $|\gamma_1|=1$) is determined by the relation 
$\Psi_1^{(1)}(0,0,ik_1)=\gamma_1\Psi_2^{(2)}(0,0,ik_1)$,
and 

$\mathbf{(iv)}\quad M(x,t,k)$  has the following behavior at $k=0$: 
\begin{itemize}
	\item in Case I,
	\begin{subequations}\label{z}
	\begin{align}
	\label{+i0}
	& M(x,t,k)=
	\begin{pmatrix}
	\frac{4}{A^2a_2(0)}v_1(x,t)& -\overline{v_2}(-x,t)\\
	\frac{4}{A^2a_2(0)}v_2(x,t)& -\overline{v_1}(-x,t)
	\end{pmatrix}
	(I+O(k))
	\begin{pmatrix}
	k& 0\\
	0& \frac{1}{k}
	\end{pmatrix}, & k\rightarrow 0,\,k\in\mathbb{C}^+\\
	\label{-i0}
	& M(x,t,k)=\frac{2i}{A}
	\begin{pmatrix}
	-\overline{v_2}(-x,t)& \frac{v_1(x,t)}{a_2(0)}\\
	-\overline{v_1}(-x,t)& \frac{v_2(x,t)}{a_2(0)}
	\end{pmatrix}
	+O(k), & k\rightarrow 0,\,k\in\mathbb{C}^-,
	\end{align}
\end{subequations}
\item 
in Case II,
\begin{subequations}\label{nz}
	\begin{align}
	\label{n+i0}
	& M(x,t,k)=
	\begin{pmatrix}
	\frac{v_1(x,t)}{a_{11}}& -\overline{v_2}(-x,t)\\
	\frac{v_2(x,t)}{a_{11}}& -\overline{v_1}(-x,t)
	\end{pmatrix}
	(I+O(k))
	\begin{pmatrix}
	1& 0\\
	0& \frac{1}{k}
	\end{pmatrix}, & k\rightarrow 0,\,k\in\mathbb{C}^+,\\
	\label{n-i0}
	& M(x,t,k)=\frac{2i}{A}
	\begin{pmatrix}
	-\overline{v_2}(-x,t)& \frac{v_1(x,t)}{a_{21}}\\
	-\overline{v_1}(-x,t)& \frac{v_2(x,t)}{a_{21}}
	\end{pmatrix}
	(I+O(k))
	\begin{pmatrix}
	1& 0\\
	0& \frac{1}{k}
	\end{pmatrix}, & k\rightarrow 0,\,k\in\mathbb{C}^-,
	\end{align}
\end{subequations}
where $v_j(x,t)$, $j=1,2$ are not specified.
\end{itemize}

\end{proposition}

\section{Asymptotics in curved wedges}\label{secas}
In \cite{RS2} we presented the long-time asymptotics of the solution of 
problem (\ref{1}) with the step-like boundary values (\ref{2.0}) along the rays $\xi=\frac{x}{4t}=const$ for all $\xi\neq0$. 
Two regions characterized by qualitatively different asymptotics 
were specified (which turn to be the same in the both Case I and Case II): 
the  decaying region  $x<0$ and the ``modulated constant'' region $x> 0$, where 
$q(x,t)\sim A\delta^2(0,\xi)$ with
\begin{equation}\label{xidel}
\delta(0,\xi)=\exp\left\{\frac{1}{2\pi i}
\int_{-\infty}^{-\xi}\frac{\ln(1+r_1(\zeta)r_2(\zeta))}{\zeta}\,d\zeta\right\}.
\end{equation}
These asymptotics do not, in general, match at $x=0$ because, as $\xi\to +0$,
\begin{equation}\label{rasdelta}
\delta(0,\xi)\sim
\begin{cases}
C_I\exp\left\{
\frac{i}{\pi}\ln\xi\cdot\ln\frac{A^2|a_2(0)|}{2\xi}
+\frac{i\ln^2\xi}{2\pi}\right\}, &
\mbox{Case I},
\\
C_{II}\exp\left\{\frac{i}{2\pi}\ln\xi\cdot\ln a_{11}a_{21}\right\},
& \mbox{Case II},
\end{cases}
\end{equation}
with some constants $C_I, C_{II}\in\mathbb{C}$ (see (\ref{asdelta}) below for the precise formula). Only in a very special case with $b(0)=0$ (which corresponds to 
 $a_{11}a_{21}=1$), the function $\delta(0,\xi)$ has a finite limit as $\xi\to0$, and the  asymptotics obtained for $\xi\ne 0$ 
 can be extended to the rays characterized by $\xi=0$
exhibiting a one-soliton behavior like (\ref{pure-sol}); 
for details, see \cite{RS2}. Otherwise, the asymptotic analysis necessitates 
introducing wedges with curved boundaries associated with slow 
variables other than $\xi$.

When introducing the new slow variable, we are guided by the idea 
that the function $\delta$, as function of this variable, 
should have a bounded modulus and (increasing) oscillations with 
well-defined principal terms and well-controlled  errors (see (\ref{delta}), (\ref{nuas}) and (\ref{rechi}) below). 
 Moreover, the new slow variable must be chosen in such a way that we are able to apply the nonlinear steepest decent method to the associated matrix Riemann-Hilbert problem (particularly, the phase function should have the structure [big parameter]$\times$ [a function depending on the slow variable and the (scaled) complex parameter only]). An important problem here is to 
provide error estimates that are smaller (in order) that 
the main asymptotic term.

We propose 
to  extend the asymptotics obtained in \cite{RS2}  into the curved wedges bounded by the parabola $t=Cx^2$, $C>0$, by introducing 
(i) the scaled spectral parameter $z=z(k)$ and 
(ii) the slow variable $s$ as follows:
\begin{equation}\label{scurve}
z=kx^{1-\alpha},\qquad s=\frac{x^{2-\alpha}}{4t},
\end{equation}
where $\alpha\in(0,1)$,  $x>0$ and $s>0$ (recall that due to (\ref{sol}) and (\ref{sol1}), the asymptotics for the RH problem for $x>0$ allows presenting
the asymptotics for $q(x,t)$ for the negative values of $x$). 
Notice that as $t\to\infty$, we have  $x\to+\infty$ for all fixed $\alpha$ 
and $s$, and 
$\xi\equiv\frac{x}{4t}=sx^{\alpha-1}$ 
decays to zero (the smaller $\alpha$, the faster is the decay) as $x\to\infty$.

In terms of $z$ and $s$, the phase function in the jump matrix $J(x,t,k)$ can be written in the form:
\begin{equation}
2ikx+4ik^2t=2ix^{\alpha}\hat{\theta}(z,s)
\end{equation}
with
\begin{equation}
\hat{\theta}(z,s)=z + \frac{z^2}{2s}.
\end{equation}
It is easy to see that $\left.\frac{\partial\hat{\theta}(z,s)}{\partial z}\right|_{z=-s}=0$, so the stationary phase point of $\hat{\theta}(z,s)$ is $z=-s$ and the signature table is similar to that for the slow variable $\xi$ (see Figure \ref{signtable}).
\begin{figure}[ht]
	\centering{\includegraphics[scale=0.04]{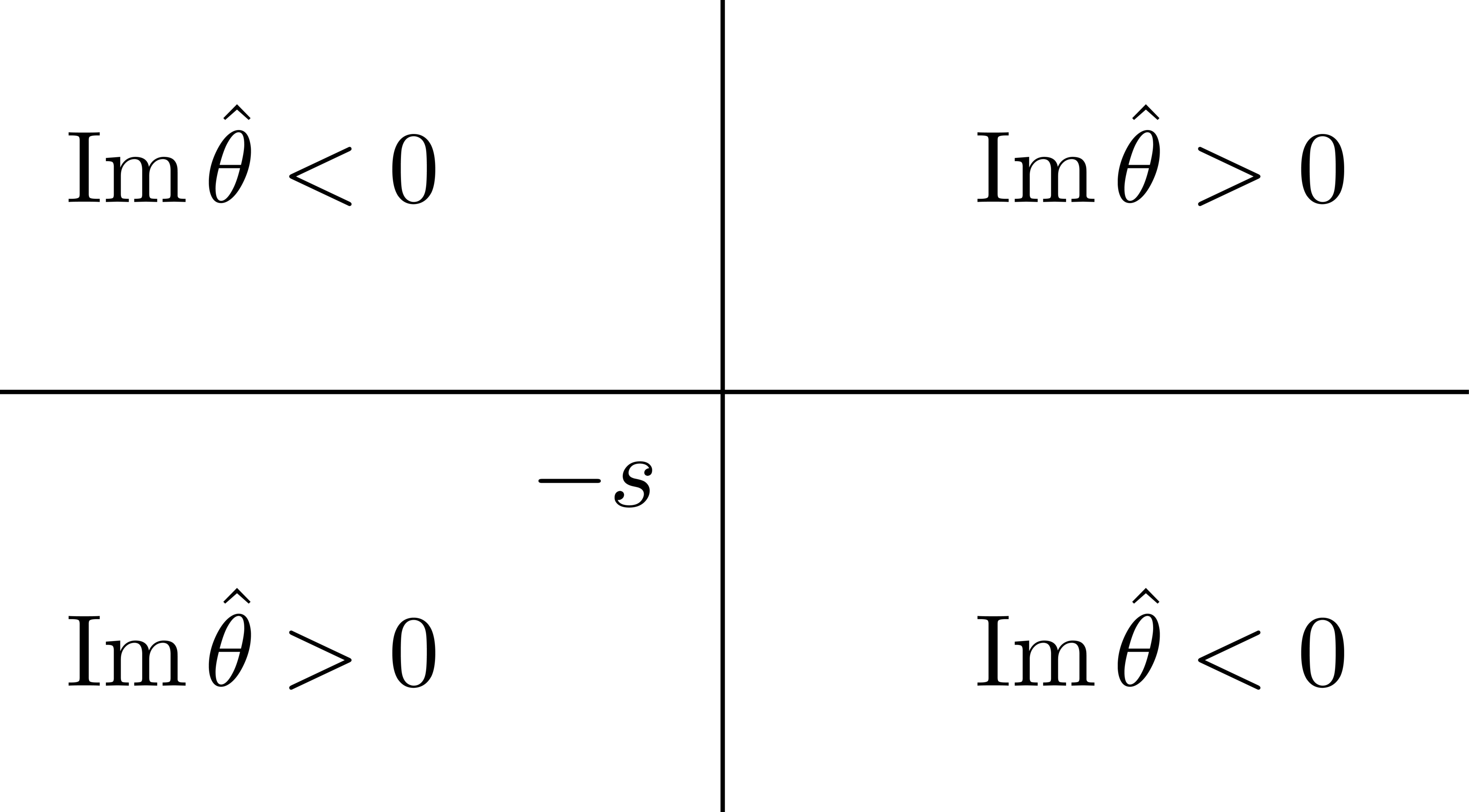}}
	\caption{Signature table of the phase function $\hat\theta(z,s)$ in the z-plane}
	\label{signtable}
\end{figure}
According to the nonlinear steepest descent method,
in order to get rid of the diagonal factors in the lower-upper triangular factorization of the jump matrix of the Riemann--Hilbert problem (see e.g. \cite{RS2, DIZ, Len15}), 
we introduce the  sectionally analytic scalar function 
(cf. (\ref{xidel}))
\begin{equation}
\hat{\delta}(z,s,t)=\exp\left\{\frac{1}{2\pi i}\int_{-\infty}^{-s}
\frac{\ln(1+\hat{r}_1(\zeta)\hat{r}_2(\zeta))}{\zeta-z}\,d\zeta
\right\},
\end{equation}
where
\begin{equation}\label{r12z}
\hat{r}_j(z):=r_j(zx^{\alpha-1})\equiv
r_j\left(z(4st)^{\frac{\alpha-1}{2-\alpha}}\right),\quad j=1,2.
\end{equation}
Integrating by parts, the function $\hat{\delta}(z,s,t)$ can be written as
\begin{equation}\label{delta}
\hat{\delta}(z,s,t)=\left(z+s\right)^{i\hat{\nu}(-s,t)}e^{\hat{\chi}(z,s,t)},
\end{equation}
where
\begin{equation}
\hat{\nu}(-s,t) = -\frac{1}{2\pi}\ln(1+\hat{r}_1(-s)\hat{r}_2(-s)),
\end{equation}
and
\begin{equation}\label{chi}
\hat{\chi}(z,s,t) = -\frac{1}{2\pi i}\int_{-\infty}^{-s}\ln(z-\zeta)\,
d_{\zeta}\ln(1+\hat{r}_1(\zeta)\hat{r}_2(\zeta)).
\end{equation}

In order to have $e^{\hat\chi(z,s,t)}$ bounded, for $z\geq-s$, 
as $t\to\infty$ (see (\ref{rechi}) below), in the present paper 
we adopt one more assumption on the behavior of the spectral functions:

\underline{\textbf{Assumption 2}}
\\
We assume that
\begin{equation}\label{arg0}
\lim\limits_{k\to 0}\left[\int_{-\infty}^{k}d_\zeta\arg(1+r_1(\zeta)r_2(\zeta))\right]=0.
\end{equation}

 Notice that in  Case I, (\ref{arg0}) is consistent with the assumption 
$\int_{-\infty}^{k}d_\zeta\arg(1+r_1(\zeta)r_2(\zeta))\in(-\pi,\pi)$, $k<0$
adopted  in \cite{RS2}, but in  Case II, it imposes an additional restriction: 
$a_{11}a_{21}>0$ or, equivalently, $|b(0)|\in[0,1)$.

\begin{proposition}\label{aschinu}
Let
$\arg(1+r_1(k)r_2(k))$ satisfy  (\ref{arg0}). Then the functions $\hat\nu(-s,t)$ and $\hat\chi(z,s,t)$ have the following behavior as $t\to\infty$ uniformly for $c\leq s\leq C$, with any fixed $C>c>0$:
\begin{equation}\label{nuas}
\hat{\nu}(-s,t) = 
\begin{cases}
\frac{1-\alpha}{\pi(2-\alpha)}\ln(4st)+\frac{1}{\pi}\ln\frac{A|a_2(0)|}{2s}
+O\left(t^{\frac{\alpha-1}{2-\alpha}}\right),&t\to\infty,\quad\mbox{Case I},\\
\frac{1}{2\pi}\ln a_{11}a_{21}(0)
+O\left(t^{\frac{\alpha-1}{2-\alpha}}\right),&t\to\infty,\quad\mbox{Case II},
\end{cases}
\end{equation}
and 
\begin{equation}\label{rechi}
\Re\hat\chi(z,s,t) = \frac{1}{2\pi}
\int_{-\infty}^{0}\frac{\arg(1+r_1(\zeta)r_2(\zeta))}{\zeta}
\,d\zeta+O\left(t^{\frac{\alpha-1}{2-\alpha}}\ln t\right),\quad z\geq-s,\quad t\to\infty
\end{equation}
for all fixed $z\geq-s$.
Moreover, in Case I we have
\begin{align}
\label{chiIa}
&\hat\chi(0,s,t)=\frac{i(1-\alpha)^2}{2\pi(2-\alpha)^2}\ln^2 4st+
\frac{i(1-\alpha)}{\pi(2-\alpha)}\ln\frac{2}{A|a_2(0)|}\cdot\ln 4st+\hat\chi_0(s)+
O\left(t^{\frac{1-\alpha}{\alpha-2}}\ln t\right),\, t\to\infty,
\\
\label{chiIb}
&\hat\chi(-s,s,t)=\frac{i(1-\alpha)^2}{2\pi(2-\alpha)^2}\ln^2 4st+
\frac{i(1-\alpha)}{\pi(2-\alpha)}\ln\frac{2}{A|a_2(0)|}\cdot\ln 4st+\hat\chi_{-s}(s)+
O\left(t^{\frac{1-\alpha}{\alpha-2}}\ln t\right),\, t\to\infty,
\end{align}
with
\begin{align}
\label{chi0}
&\hat\chi_0(s)=\frac{i\ln^2 s}{2\pi} 
+\frac{i}{2\pi}\int_{-\infty}^{-1}\ln(-\zeta)\,d_{\zeta}\ln(1+r_1(\zeta)r_2(\zeta))
+\frac{i}{2\pi}\int_{-1}^{0}\ln(-\zeta)\,
d_{\zeta}\ln\frac{1+r_1(\zeta)r_2(\zeta)}{\zeta^2},
\\
&\hat\chi_{-s}(s) = \hat\chi_0(s)+\frac{\pi i}{6},
\end{align}
and in Case II, for all fixed $z\geq-s$,
\begin{equation}\label{chiII}
\hat\chi(z,s,t)=\frac{i(1-\alpha)}{2\pi(\alpha-2)}\ln(a_{11} a_{21})\cdot
\ln 4st+\hat\chi_1+O\left(t^{\frac{1-\alpha}{\alpha-2}}\ln t\right),\quad t\to\infty,
\end{equation}
where
\begin{equation}\label{chi1}
\hat\chi_1=\frac{i}{2\pi}\int_{-\infty}^{0}\ln(-\zeta)\,
d_{\zeta}\ln(1+r_1(\zeta)r_2(\zeta)).
\end{equation}
\end{proposition}
\begin{proof}
	See  Appendix A.
\end{proof}

In order to make the derivation of our main result more transparent,
in what follows we adopt the following technical 

\underline{\textbf{Assumption 3}} 

We assume that 
 the reflection coefficients $r_j(k)$, $j=1,2$ can be analytically continued into the whole complex plane. 

This  takes place, for example, when the initial data are such that $q_0(x)=0$ for $x<-R$ and $q_0(x)=A$ for $x>R$, with some $R>0$. Otherwise,  
techniques of analytical approximations of the reflection coefficients 
and associated error estimates
can be applied, see, e.g.,  \cite{DZ} and \cite{Len15}. 

The formalism of the nonlinear steepest descent method is based
on subsequent transformations of the basic Riemann--Hilbert problem, 
aimed at arriving (after appropriate rescaling, if needed) at 
some model RH problem that can be solved explicitly. Particularly,
having in the phase function a large parameter multiplied by a quadratic polynomial
with respect of the (scaled) spectral variable would lead
to the model problem that can be solved in terms of the parabolic cylinder functions
\cite{I1, DZ, Len15}. The specific feature of the analysis
in the present paper, as we will see, is that the parameters of the parabolic cylinder functions are not constants but depend on the fast variable $t$.

The series of transformations of the basic RH problem (for $M(x,t,k)$) is similar to that made in \cite{RS2}, so here we only briefly recall the definition of new sectionally meromorphic matrix functions (for more details see the corresponding transformations in \cite{RS2}). The first transformation enables us to get rid of the diagonal factors in the lower-upper triangular factorization:
\begin{equation}\label{ft}
\tilde{M}(x,t,z)=M(x,t,k(z))\hat\delta^{-\sigma_3}(z,s,t).
\end{equation}
Next, in order to ``get off'' the real axis we introduce the matrix $\hat M(x,t,z)$ as follows:
(see Figure \ref{mod1s} where domains $\hat\Omega_j$, $j=\overline{0,4}$ are defined)
\begin{figure}[ht]
	\centering{\includegraphics[scale=0.12]{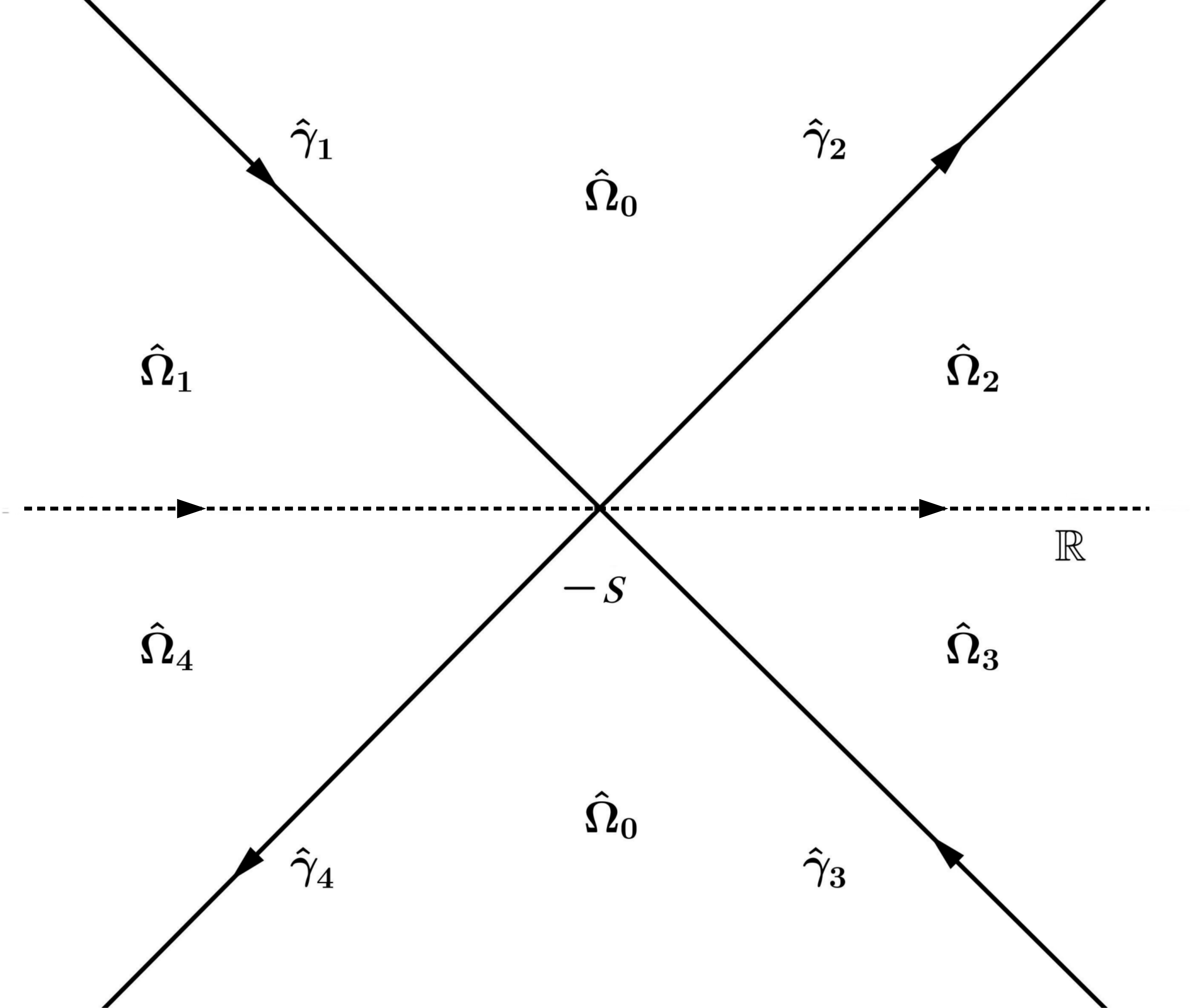}}
	\caption{Contour $\hat\Gamma=\bigcup\limits_{j=1}^{4}\hat\gamma_j$ and domains $\hat\Omega_j$, $j=\overline{0,4}$ for $\hat M(x,t,z)$ in the $z$-plane.}
	\label{mod1s}
\end{figure}
\begin{equation}
\hat{M}(x,t,z)=
\begin{cases}
\tilde{M}(x,t,z),& z\in\hat\Omega_0,\\
\tilde{M}(x,t,z)
\begin{pmatrix}
1& \frac{-\hat r_2(z)\hat\delta^{2}(z,s,t)}
{1+\hat r_1(z)\hat r_2(z)}e^{-2ix^{\alpha}\hat\theta}\\
0& 1\\
\end{pmatrix}
,& z\in\hat\Omega_1,
\\
\tilde{M}(x,t,z)
\begin{pmatrix}
1& 0\\
-\hat r_1(z)\hat\delta^{-2}(z,s,t)e^{2ix^{\alpha}\hat\theta}& 1\\
\end{pmatrix}
,& z\in\hat\Omega_2,
\\
\tilde{M}(x,t,z)
\begin{pmatrix}
1& \hat r_2(z)\hat\delta^2(z,s,t)e^{-2ix^{\alpha}\hat\theta}\\
0& 1\\
\end{pmatrix}
,& z\in\hat\Omega_3,
\\
\tilde{M}(x,t,z)
\begin{pmatrix}
1& 0\\
\frac{\hat r_1(z)\hat\delta^{-2}(z,s,t)}{1+\hat r_1(z)\hat r_2(z)}e^{2ix^{\alpha}\hat\theta}& 1\\
\end{pmatrix}
,& z\in\hat\Omega_4.
\end{cases}
\end{equation}
Matrix $\hat M(x,t,z)$ solves the RH problem on the cross $\hat\Gamma=\bigcup\limits_{j=1}^{4}\hat\gamma_j$, where $\hat\gamma_j=\{-s+te^{i(\frac{5\pi}{4}-\frac{\pi j}{2})}|t\in[0,\infty)\}$ (see Figure \ref{mod1s}) which is characterized by the jump conditions
\begin{subequations}\label{RHhatz}
\begin{equation}
\hat{M}_+(x,t,z)=\hat{M}_-(x,t,z)\hat{J}(x,t,z),\qquad k\in\hat\Gamma,
\end{equation}
with 
\begin{equation}
\label{J-hatz}
\hat{J}(x,t,z)=
\begin{cases}
\begin{pmatrix}
1& \frac{\hat r_2(z)\hat \delta^{2}(z,s,t)}{1+\hat r_1(z)\hat r_2(z)}
e^{-2ix^{\alpha}\hat \theta}\\
0& 1\\
\end{pmatrix}
,& z\in\hat\gamma_1,
\\
\begin{pmatrix}
1& 0\\
\hat r_1(z)\hat \delta^{-2}(z,s,t)e^{2ix^{\alpha}\hat \theta}& 1\\
\end{pmatrix}
,& z\in\hat\gamma_2,
\\
\begin{pmatrix}
1& -\hat r_2(z)\hat \delta^2(z,s,t)e^{-2ix^{\alpha}\hat \theta}\\
0& 1\\
\end{pmatrix}
,& z\in\hat\gamma_3,
\\
\begin{pmatrix}
1& 0\\
\frac{-\hat r_1(z)\hat \delta^{-2}(z,s,t)}{1+\hat r_1(z)\hat r_2(z)}
e^{2ix^{\alpha}\hat \theta}& 1\\
\end{pmatrix}
,& z\in\hat\gamma_4,
\end{cases}
\end{equation}
the normalization condition
\begin{equation}
\hat{M}(x,t,z)\rightarrow I, \qquad z\rightarrow\infty,
\end{equation}
and the residue conditions
\begin{equation}
\label{15.3}
\underset{z=ik_1x^{1-\alpha}}{\operatorname{Res}} \hat{M}^{(1)}(x,t,z)=
\hat c_1(x,t)\hat{M}^{(2)}(x,t,ik_1x^{1-\alpha}),
\end{equation}
\begin{equation}
\label{res-convent}
\underset{z=0}{\operatorname{Res}}\  \hat{M}^{(2)}(x,t,z)=
\hat c_0(x,t)\hat{M}^{(1)}(x,t,0),
\end{equation}
\end{subequations}
where $\hat c_1(x,t)=\frac{\gamma_1 x^{1-\alpha}}
{\dot{a}_1(ik_1)\hat \delta^{2}(ik_1x^{1-\alpha},s)}
e^{-2k_1x-4ik_1^2t}$ 
and $\hat c_0(x,t)=\frac{A}{2i}x^{1-\alpha}\hat\delta^2(0,s,t)$.

The solution of the original initial value problem can be represented, in terms 
of the solution $\hat{M}(x,t,z)$ of this RH problem, as follows:
\begin{equation}
q(x,t) = 2ix^{\alpha-1}\lim\limits_{z\to\infty}z\hat M(x,t,z),\quad x>0,
\end{equation}
and
\begin{equation}
q(-x,t) = -2ix^{\alpha-1}\lim\limits_{z\to\infty}z\overline{\hat M(x,t,z)}, \quad x>0.
\end{equation}

The (singular)  RH problem (\ref{RHhatz}), which involves the residue conditions, can be transformed to a regular one  by using the Blashke--Potapov factors \cite{FT}:
\begin{proposition}(cf. \cite{RS2})
\label{propsolz}
The solution $q(x,t)$ can be represented as follows:
\begin{subequations}
\label{sol+0z}
\begin{align}
\label{sol+0az}
q(x,t)&=-2k_1P_{12}(x,t)+
2ix^{\alpha-1}\lim\limits_{z\to\infty}z\hat{M}^{R}_{12}(x,t,z),
\quad x>0,\\
q(-x,t)&=-2k_1\overline{P_{21}(x,t)}
-2ix^{\alpha-1}\lim\limits_{z\to\infty}
z\overline{\hat{M}^{R}_{21}(x,t,z)},\quad x>0.
\end{align}
\end{subequations}
Here $\hat{M}^{R}(x,t,k)$ solves the regular Riemann-Hilbert problem:
\begin{subequations}\label{RHRz}
\begin{equation}
\begin{cases}
\hat{M}^R_+(x,t,z)=\hat{M}^R_-(x,t,z)\hat{J}^R(x,t,z),& k\in\hat\Gamma,\\
\hat{M}^R(x,t,z)\rightarrow I, & k\rightarrow\infty,
\end{cases}
\end{equation}
with 
\begin{equation}
\label{J^Rz}
\hat{J}^R(x,t,z)=
\begin{pmatrix}
1& 0\\
0& \frac{z-ik_1x^{1-\alpha}}{z}
\end{pmatrix}
\hat{J}(x,t,z) 
\begin{pmatrix}
1& 0\\
0& \frac{z}{z-ik_1x^{1-\alpha}}
\end{pmatrix},\quad k\in\hat{\Gamma},
\end{equation}
\end{subequations}
and $P_{12}$ and $P_{21}$ are determined in terms of $\hat M^R$ as follows:
\begin{equation}
\label{P}
P_{12}(x,t)=\frac{g_1(x,t)h_1(x,t)}
{g_1(x,t)h_2(x,t)-g_2(x,t)h_1(x,t)},\,
P_{21}(x,t)=-\frac{g_2(x,t)h_2(x,t)}
{g_1(x,t)h_2(x,t)-g_2(x,t)h_1(x,t)},
\end{equation}
where $g(x,t)=\left(
\begin{smallmatrix}g_1(x,t)\\g_2(x,t) \end{smallmatrix}\right)$
and  $h(x,t)=\left(
\begin{smallmatrix}h_1(x,t)\\h_2(x,t) \end{smallmatrix}\right)$
are given by
\begin{subequations}\label{ghz}
\begin{align}
g(x,t)&=ik_1x^{1-\alpha}\hat{M}^{R(1)}(x,t,ik_1x^{1-\alpha})-
\hat c_1(x,t)\hat{M}^{R(2)}(x,t,ik_1x^{1-\alpha}),\\
h(x,t)&=ik_1x^{1-\alpha}\hat{M}^{R(2)}(x,t,0)+
\hat c_0(x,t)\hat{M}^{R(1)}(x,t,0).
\end{align}
\end{subequations}
\end{proposition}

\begin{cor}
The rough asymptotics of $q(x,t)$ as $t\to\infty$ along the curve $t=\frac{x^{2-\alpha}}{4s}$ with  fixed $s>0$ and $\alpha\in(0,1)$ has the form:
$$
q(x,t)=A\hat\delta^2(0,s,t) + o(1),\quad x>0,\qquad q(-x,t)=o(1),\quad x>0.
$$
Taking into account (\ref{delta}) and Proposition \ref{aschinu}, the asymptotics  for $x>0$ has the form
\begin{equation}\label{rasc}
q(x,t)=
\begin{cases}
\mathrm{Q} e^{i\Psi_I(\alpha, s, t)} + o(1),& t\to\infty,\quad
\mbox{Case I},\\
\mathrm{Q}e^{i\Psi_{II}(\alpha, s, t)} + o(1),& 
t\to\infty,\quad\mbox{Case II},
\end{cases}
\end{equation}
where (notice that $2\Re\hat\chi(0,s,t)\sim
\frac{1}{\pi}\int_{-\infty}^{0}\frac{\arg(1+r_1(\zeta)r_2(\zeta))}
{\zeta}\,d\zeta$ for both Case I and Case II, see (\ref{rechi}))
\begin{equation}\label{AIAII}
\mathrm{Q}=A\exp\left\{\frac{1}{\pi}\int_{-\infty}^{0}
\frac{\arg(1+r_1(\zeta)r_2(\zeta))}{\zeta}\,d\zeta\right\},
\end{equation}
and (see (\ref{nuas}), (\ref{chiIa}) and (\ref{chiII}))
\begin{subequations}\label{PsiI-II}
\begin{align}
\label{PsiI}
&\Psi_I(\alpha, s, t)=
\psi(\alpha)\ln^2 4st+\phi_I(\alpha,s)\ln 4st+
\frac{2}{\pi}\ln s\cdot\ln\frac{A|a_2(0)|}{2s}+2\Im\hat\chi_0(s),\\
&\Psi_{II}(\alpha, s, t)=
\phi_{II}(\alpha)\ln 4st+
\frac{1}{\pi}\ln s\cdot\ln a_{11}a_{21}
+2\Im\hat\chi_1,
\end{align}	
\end{subequations}
with
\begin{align}
\label{principle-const}
&\psi(\alpha)=\frac{(1-\alpha)^2}{\pi(2-\alpha)^2},\quad
\phi_I(\alpha, s)=\frac{2(1-\alpha)}{\pi(2-\alpha)}\ln\frac{2s}{A|a_2(0)|},\quad
\phi_{II}(\alpha)=\frac{1-\alpha}{\pi(\alpha-2)}\ln a_{11}a_{21},
\end{align}
and $\hat\chi_0(s)$ and $\hat\chi_1$  given by (\ref{chi0}) and (\ref{chi1}) respectively.
\end{cor} 
\begin{remark}
Here and below, we prefer to keep  $4st$ as the argument of the logarithms since 
(i) in terms of this expression, the coefficients (\ref{principle-const}) have simpler form and 
(ii) it is directly related to $x$  (see (\ref{scurve})), which 
is convenient when studying a domain surrounding 
the ray $x=0$.
\end{remark}
\begin{remark}
The behavior of the solution along the curves $t=\frac{x^{2-\alpha}}{4s}$, $s=const$ in Cases I and II are different; this is in contrast 
with the asymptotics along straight lines  $t=\frac{x}{4\xi}$, $\xi=const$ 
\cite{RS2}, which has the same form in the both cases.
\end{remark}
\begin{remark}
In the reflectionless case (i.e.,  when $b(k)= 0$ for all $k\in\mathbb R$ and 
$q(x,t)$ is a one-soliton solution (\ref{pure-sol})), the main term in  (\ref{rasc}) in Case II is equal to $A$, which is consistent with the explicit soliton formula 
(\ref{pure-sol}).
\end{remark}

\begin{figure}[ht]
	\centering{\includegraphics[scale=18.0]{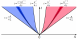}}
	\caption{The asymptotic wedges for $q(x,t)$
specified  by $\alpha\in(0,1)$ and consisting of the curves $t=\frac{|x|^{2-\alpha}}{4s}$ with $s$ varying from $0$ to $\infty$. 
Single wedges (with fixed $\alpha$ and varying $s$) are shown as darker regions.
In the left integral wedge, the solution decays to $0$;  in the right wedge, it has the form of non-vanishing oscillations.}
	\label{regions}
\end{figure}

Applying the nonlinear steepest decent method \cite{DZ} allows making this asymptotics more precise.
\begin{theorem}\label{th1}
Assume that the spectral functions associated with the 
the initial data $q_0(x)$ satisfy Assumptions 1-3, and let $\alpha\in(0,1)$
be fixed.
Then 
 the long-time asymptotics of the solution $q(x,t)$ of problem (\ref{1}), 
(\ref{2.0}) 
along the curve $t=\frac{x^{2-\alpha}}{4s}$ with any $s>0$ fixed 
(see Figure \ref{regions}) has the form:
\\
\textbf{Case I:}
\begin{subequations}\label{caseI}
\begin{align}\label{asIg0}
&q(x,t)=
\begin{cases}
\mathrm{Q}e^{i\Psi_I(\alpha,s,t)}+t^{\frac{\alpha}{2\alpha-4}}\sqrt{\ln t}\,F_I(\alpha,s,t)
+O\left(\frac{t^{\frac{\alpha}{2\alpha-4}}}{\sqrt{\ln t}}\right),
&x>0,\,\alpha\in\left(0,\frac{2}{3}\right),\\
\mathrm{Q}e^{i\Psi_I(\alpha,s,t)}
+O\left(t^{\frac{1-\alpha}{\alpha-2}}\ln t\right),
&x>0,\,\alpha\in\left[\frac{2}{3}, 1\right),
\end{cases}
\\
\label{asIl0}
&q(-x,t)=
\begin{cases}
O\left(t^{\frac{1}{\alpha-2}}{\ln t}\right),
&x>0,\,\alpha\in\left(0, \frac{2}{3}\right],\\
t^{\frac{4-3\alpha}{2\alpha-4}}\sqrt{\ln t}A_{I,\,3}^{as}(\alpha, s)
e^{i\Psi_{I,1}(\alpha,s,t)}
+O\left(\frac{t^{\frac{4-3\alpha}{2\alpha-4}}}{\sqrt{\ln t}}\right),
&x>0,\,\alpha\in\left(\frac{2}{3},1\right),
\end{cases}
\end{align}
\end{subequations}
\textbf{Case II:}
\begin{subequations}\label{caseII}
\begin{align}\label{asIIg0}
&q(x,t)=
\begin{cases}
\mathrm{Q}e^{\Psi_{II}(\alpha, s, t)}+
t^{\frac{\alpha}{2\alpha-4}}F_{II}(\alpha,s,t)+R_1(\alpha,t),
&x>0,\,\alpha\in\left(0,\frac{2}{3}\right),\\
\mathrm{Q}e^{\Psi_{II}(\alpha, s, t)}
+O\left(t^{\frac{1-\alpha}{\alpha-2}}\ln t\right),
&x>0,\,\alpha\in\left[\frac{2}{3},1\right),
\end{cases}\\
\label{asIIl0}
&q(-x,t)=
\begin{cases}
O\left(t^{\frac{1}{\alpha-2}}{\ln t}\right),
&x>0,\,\alpha\in\left(0, \frac{2}{3}\right],\\
t^{\frac{4-3\alpha}{2\alpha-4}}A_{II,\,3}^{as}(\alpha, s)e^{i\Psi_{II,1}(\alpha,s,t)}
+R_2(\alpha,t),
&x>0,\,\alpha\in\left(\frac{2}{3},1\right),
\end{cases}
\end{align}
\end{subequations}
where $\mathrm{Q}$ is given by (\ref{AIAII}),
 $\Psi_{I}$ and  $\Psi_{II}$ are given by (\ref{PsiI-II}),
\begin{align}\label{F_I}
&F_I(\alpha,s,t)=
A_{I,\,1}^{as}(\alpha, s)e^{i\Psi_{I,1}(\alpha,s,t)}
+A_{I,\,2}^{as}(\alpha, s)e^{i\Psi_{I,2}(\alpha,s,t)},\\
\label{F_II}
&F_{II}(\alpha,s,t)=
A_{II,\,1}^{as}(\alpha, s)e^{i\Psi_{II,1}(\alpha,s,t)}
+A_{II,\,2}^{as}(\alpha, s)e^{i\Psi_{II,2}(\alpha,s,t)},
\end{align}
with
\begin{equation}
A_{j,\,1}^{as}(\alpha, s)=-\frac{2k_1}{s}\tilde\beta^R_{j,\,as}(\alpha,s),\,
A_{j,\,2}^{as}(\alpha, s)=\frac{A_{j}^2(s)}{2k_1s}
\tilde\gamma^R_{j,\,as}(\alpha,s),\,
A_{j,\,3}^{as}(\alpha, s)=\frac{s^\frac{\alpha}{2-\alpha}}
{2^\frac{2-3\alpha}{2-\alpha}k_1}
\overline{\tilde\gamma^R_{j,\,as}(\alpha,s)},
\end{equation}
where $\tilde\beta^R_{j,\,as}$ and $\tilde\gamma^R_{j,\,as}$, $j=I,II$ are given by (\ref{tilde-be-ga-as}) and
\begin{align}
\nonumber
&&\Psi_{I,j}(\alpha,s,t)=&(-1)^{j+1}
\phi_0(\alpha, s)t^{\frac{\alpha}{2-\alpha}}
+\phi_{1j}(\alpha)\ln^24st
+(-1)^{j+1}\phi_2(\alpha)\ln 4st\cdot\ln\ln 4st\\
&& &+\phi_{3j}(\alpha,s)\ln4st
+ (-1)^{j+1}\phi_4(s)\ln\ln 4st, \quad j=1,2,\\
&&\Psi_{II,j}(\alpha,s,t)=&(-1)^{j+1}\phi_0(\alpha, s)t^{\frac{\alpha}{2-\alpha}}
+ \phi_{5j}(\alpha)\ln 4st,\quad j=1,2.
\end{align}
with
\begin{subequations}\label{phij1}
\begin{align}
\label{phij}
&\phi_0(\alpha,s)=2^{\frac{2\alpha}{2-\alpha}}s^{\frac{2}{2-\alpha}},\quad
\phi_{11}(\alpha)=\frac{(1-\alpha)(1-2\alpha)}{\pi(2-\alpha)^2},\quad
\phi_{12}(\alpha)=\frac{1-\alpha}{\pi(2-\alpha)^2},\\
&\phi_{31}(\alpha,s)=\frac{1-\alpha}{\pi(2-\alpha)}
\left(
\ln\frac{1-\alpha}{\pi(2-\alpha)}+\ln\frac{2s}{A^2a_2^2(0)}
+\frac{\alpha}{\pi(1-\alpha)}\ln\frac{2s}{A|a_2(0)|}-1
\right),
\\
&\phi_{32}(\alpha,s)=\frac{1-\alpha}{\pi(2-\alpha)}
\left(
\ln\frac{\pi(2-\alpha)}{1-\alpha}+\ln\frac{8s^3}{A^2a_2^2(0)}
+\frac{\alpha}{\pi(1-\alpha)}\ln\frac{A|a_2(0)|}{2s}+1
\right),\\
&\phi_2(\alpha)=\frac{1-\alpha}{\pi(2-\alpha)},\,
\phi_4(s)=\frac{1}{\pi}\ln\frac{A|a_2(0)|}{2s},\,
\phi_{51}(\alpha)=\frac{\ln a_{11}a_{21}}{2\pi(\alpha-2)},\,
\phi_{52}(\alpha)=(4\alpha-5)\phi_{51}(\alpha),
\end{align}
\end{subequations}
and
\begin{equation}
R_1(\alpha,t)=
\begin{cases}
O\left(t^{\frac{\alpha}{\alpha-2}}\ln t\right),
&\alpha\in\left(0,\frac{1}{2}\right),\\
O\left(t^{\frac{1-\alpha}{\alpha-2}}\ln t\right),
&\alpha\in\left[\frac{1}{2},\frac{2}{3}\right),
\end{cases}
\quad
R_2(\alpha,t)=
\begin{cases}
O\left(t^{\frac{1}{\alpha-2}}\ln t\right),
&\alpha\in\left(\frac{2}{3},\frac{4}{5}\right],\\
O\left(t^{\frac{6-5\alpha}{2\alpha-4}}\sqrt{\ln t}\right),
&\alpha\in\left(\frac{4}{5},1\right).
\end{cases}
\end{equation}
\end{theorem}
\begin{remark}
The error estimates  in the asymptotic  formulas 
arise as sums of two terms: (i) the error estimate in the expansion of $\hat\nu(-x,t)$ and $\hat\chi(0,s,t)$, which is of order $O\left(t^{\frac{1-\alpha}{\alpha-2}}\ln t\right)$ (see Proposition \ref{aschinu}), and (ii) the error estimate
 in the expansion of the local parametrix (see (\ref{asparam})).
\begin{itemize}
	\item 
	For the error estimates in (\ref{asIg0}) and  (\ref{asIIg0}), 
	we have $\frac{1-\alpha}{\alpha-2}<\frac{\alpha}{2\alpha-4}$ for 
	$0<\alpha<\frac{2}{3}$ and thus the overall estimates for $\frac{2}{3}\leq\alpha<1$ are due to the (dominating) estimates for $\hat\nu$ and $\hat\chi$.   
	Expanding $\hat\nu(-s,t)$ and $\hat\chi(0,s,t)$ 
	(see (\ref{nuas}) and (\ref{chiIa})) up to $O\left(t^{\frac{n-n\alpha}{\alpha-2}}\ln t\right)$ for some $n\in\mathbb{N}$, the first decaying terms in 
	(\ref{asIg0}) and  (\ref{asIIg0}) can be specified explicitly for  all 
	$\alpha$ up to $\frac{2n}{2n+1}$.
	\item
	For the error estimates in (\ref{asIl0}) and (\ref{asIIl0}), the main contribution comes from the local parametrix (\ref{asparam})
	(notice that $\frac{4-3\alpha}{2\alpha-4}<\frac{1}{\alpha-2}$ 
	for $0<\alpha<\frac{2}{3}$ as above). 
	In order to obtain the first asymptotic terms in (\ref{asIl0}) and  (\ref{asIIl0}) for $0<\alpha\leq\frac{2}{3}$ (which is of particular interest since the smaller $\alpha$, the faster the asymptotic curve 
 approaches the ray $x=0$),   the higher order theory \cite{DZ94} is to be used 
for obtaining better estimates in the expansion of the local parametrix. 
This problem is technically involved and will be addressed elsewhere;
here we notice that   for $0<\alpha\leq\frac{2}{3}$, the solution decays apparently
as $t^{\tilde{p}(\alpha)}\ln^{\tilde{q}} t$, where $\tilde{p}(\alpha)<0$ can be written explicitly and $\tilde{q}>\frac{1}{2}$ (cf. \cite{DZ94} in the NLS case). Consequently, there should be region(s) for $0<\alpha\leq\frac{2}{3}$ with qualitatively different decaying regimes comparing with that for $\frac{2}{3}<\alpha<1$.
\end{itemize}
\end{remark}
\begin{proof}
First, introduce the rescaled spectral parameter $w$ by
\begin{equation}\label{k-z}
z=\sqrt{\frac{s}{2x^{\alpha}}}w-s,
\end{equation}
so that 
\[
e^{2ix^{\alpha}\hat \theta} = e^{\frac{iw^2}{2}-isx^{\alpha}}.
\]
Next, introduce the ``local parametrix'' $\hat m_0^R(x,t,z)$
as the solution of the RH problem with the jump matrix, 
which  is a ``simplified version of $\hat{J}^{R}(x,t,z)$''
in the sense that in its construction, $\hat r_j(z)$, $j=1,2$ are replaced by the constants $\hat r_j(-s)$,
and $\hat \delta(z,s,t)$ is replaced by
$\left(\sqrt{\frac{s}{2x^{\alpha}}}w\right)
^{i\hat\nu(-s,t)}e^{\hat\chi\left(-s,s,t\right)}$. 
This RH problem can be solved explicitly, 
in terms of the parabolic cylinder functions \cite{I1}.

Indeed, $\hat m_0^R(x,t,z)$ can be determined by 
\begin{equation}\label{m0-R}
\hat m_0^R(x,t,z)=\Delta(x,t)m^{\Gamma}(s,w(z))\Delta^{-1}(x,t),
\end{equation}
where 
\begin{equation}\label{Delta}
\Delta(x,t) = e^{\left(isx^{\alpha}/2 + \hat\chi(-s,s,t)
\right)\sigma_3}\left(\frac{2x^{\alpha}}{s}\right)^{-\frac{i\hat\nu(-s,t)}{2}\sigma_3},
\end{equation}
and $m^{\Gamma}(s,w,t)$ is determined in the same terms as in \cite{RS2}, namely,
\begin{equation}\label{m-g-0}
m^{\Gamma}(s,w,t) = m_0(s,w,t) D^{-1}_{j}(s,w,t),\qquad w\in\Omega_j,\,\,j=0,\ldots,4,
\end{equation}
see Figure \ref{mod2},
where $\gamma_j$ corresponds to $\hat\gamma_j$ in view of (\ref{k-z}) with
$D_0(s,w,t)=e^{-i\frac{w^2}{4}\sigma_3}w^{i\hat\nu(-s,t)\sigma_3}$ and
\[
\begin{aligned}
& D_1(s,w,t)=D_0(s,w,t)
\begin{pmatrix}
1& \frac{\hat r_2^R(-s)}{1+ \hat r_1^R(-s)\hat r_2^R(-s)}\\
0& 1\\
\end{pmatrix},
& \qquad & 
D_2(s,w,t)=D_0(s,w,t)
\begin{pmatrix}
1& 0\\
\hat r_1^R(-s)& 1\\
\end{pmatrix}, \\
& D_3(s,w,t)=D_0(s,w,t)
\begin{pmatrix}
1& - \hat r_2^R(-s)\\
0& 1\\
\end{pmatrix},
& &
D_4(s,w,t)=D_0(s,w,t)
\begin{pmatrix}
1& 0\\
\frac{-\hat r_1^R(-s)}{1+ \hat r_1^R(-s)\hat r_2^R(-s)}& 1
\end{pmatrix},
\end{aligned}
\]
with (see (\ref{J^Rz}))
\begin{equation}\label{r12R}
\hat r_1^R(z)= \frac{z-ik_1x^{1-\alpha}}{z}\hat r_1(z),\quad 
\hat r_2^{R}(z)=\frac{z}{z-ik_1x^{1-\alpha}}\hat r_2(z).
\end{equation}
In turn,  $m_0(s,w,t)$ is the solution of the  RH problem with a \emph{constant} 
(w.r.t. the spectral parameter $w$) jump matrix:
\begin{equation}\label{as8}
\begin{cases}
m_{0+}(s,w,t)=m_{0-}(s,w,t)j_0(s),& w\in\mathbb{R},\\
m_0(s,w,t)= \left(I+O(1/w)\right)
e^{-i\frac{w^2}{4}\sigma_3}w^{i\hat \nu(-s,t)\sigma_3},& w\rightarrow\infty,
\end{cases}
\end{equation}
with  
\begin{equation}\label{j0}
j_0(s)=
\begin{pmatrix}
1+\hat r_1^R(-s)\hat r_2^R(-s) & \hat r_2^R(-s)\\
\hat r_1^R(-s) & 1
\end{pmatrix}
\end{equation}
 (we drop the dependence on 
$\alpha$), which can be solved explicitly in terms of the parabolic cylinder 
functions \cite{I1}.

\begin{figure}[ht]
	\centering{\includegraphics[scale=0.1]{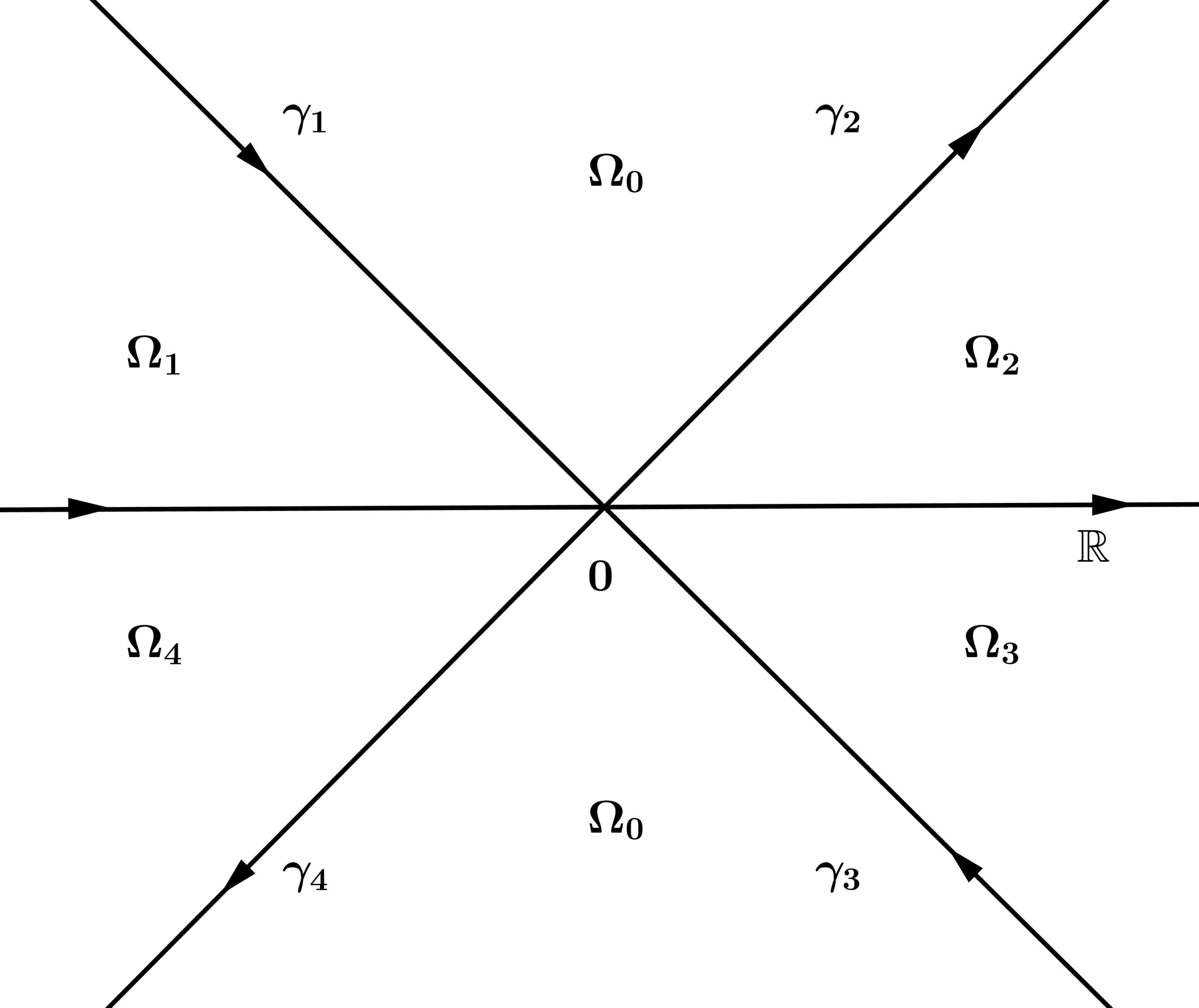}}
	\caption{Contour and domains for $m^{\Gamma}(s,w)$ in the $w$-plane  }
	\label{mod2}
\end{figure}

Since our goal is to obtain the long-time asymptotics, 
we only need from $m_0(s,w,t)$ (and
$m^{\Gamma}(s,w,t)$) its large-$w$ asymptotics, which has the form
\[
m^{\Gamma}(s,w,t) = I + \frac{i}{w}\begin{pmatrix}
0 & \beta^R(s,t) \\ -\gamma^R(s,t) & 0
\end{pmatrix} + O(w^{-2}), \qquad w\to \infty,
\]
where 
\begin{subequations}\label{be-ga-Rz}
\begin{align}
\beta^R(s,t)=\dfrac{\sqrt{2\pi}e^{-\frac{\pi}{2}\hat\nu(-s,t)}
e^{-\frac{3\pi i}{4}}}{\hat r_1^R(-s)\Gamma(-i\hat\nu(-s,t))},\\
\gamma^R(s,t)=\dfrac{\sqrt{2\pi}e^{-\frac{\pi}{2}\hat\nu(-s,t)}
e^{-\frac{\pi i}{4}}}{\hat r_2^R(-s)\Gamma(i\hat\nu(-s,t))}.
\end{align}
\end{subequations}

Now, having defined the parametrix $\hat m_0^R(x,t,z)$,
we define $\check M^{R}(x,t,z)$ 
as follows:
\[
\check M^R(x,t,z) = 
\begin{cases}
\hat M^R(x,t,z)(\hat m_0^{R})^{-1}(x,t,z), & |z+s|<\varepsilon, \\
\hat M^R(x,t,z), & \mbox{otherwise},
\end{cases}
\]
where $\varepsilon$ is small enough so that $|s|>\varepsilon$ and $|ik_1x^{1-\alpha}+s|>\varepsilon$. Then the sectionally analytic matrix $\check M^{R}$ has the following jumps across  
$\hat\Gamma_1=\hat\Gamma\cup\{|z+s|=\varepsilon\} $ (the circle 
$|z+s|=\varepsilon $ is  oriented counterclockwise)
\begin{equation}\label{check-J}
\check J^R(x,t,z) = 
\begin{cases}
\hat m_{0-}^R(x,t,z) \hat J^R(x,t,z)(\hat m_{0+}^R)^{-1} (x,t,z)
, & z\in \hat\Gamma, |z+s|<\varepsilon, \\
\left(\hat m_{0}^{R}\right)^{-1} (x,t,z), & |z+s|=\varepsilon, \\
\hat J^R(x,t,z), & \text{otherwise}.
\end{cases}
\end{equation}

Using the representation of $\check M^{R}(x,t,z)$ in terms of the solution of the  singular integral equation, we obtain its long time behavior (see (\ref{M^R}) below). 
The integral representation has the form:
\begin{equation}\label{M-int-rep}
\check M^{R}(x,t,z) = I+\frac{1}{2\pi i}\int_{\hat\Gamma_1}\mu(x,t,\zeta)(\check J^{R}(x,t,\zeta)-I)\frac{d\zeta}{\zeta-z},
\end{equation}
where $\mu$ solves the integral equation $\mu -C_u \mu = I$, with 
$u=\check J^{R} - I$.
Here the Cauchy-type operator $C_u$ is defined by 
$C_u f = C_-(fu)$, where 
$(C_-h)(z)$, $z\in \hat\Gamma_1$ are the right (according to the orientation of 
$\hat\Gamma_1$) non-tangential boundary values of 
\[
(Ch)(z')=\frac{1}{2\pi i}\int_{\hat\Gamma_1}\frac{h(\zeta)}{\zeta-z'}\,d\zeta, \quad z'\in{\mathbb C}\setminus
\hat\Gamma_1.
\]
Since $\Im\hat\nu(-s,t)=o(1)$ as $t\to\infty$ (see (\ref{nuas})),
it follows that  $\Im\hat\nu(-s,t)$ does not affect the power exponent in the decaying term, which is in contrast with the case of constant $\xi$ \cite{RS, RS2}. Thus the large-$t$ asymptotics of $(\hat m_0^{R})^{-1}$ has the form:
\begin{align}
\nonumber
(\hat m_0^{R})^{-1}(x,t,z)&=\Delta(x,t)
\left(m^{\Gamma}\right)^{-1}(s,\sqrt{2x^{\alpha}/s}(z+s),t)\Delta^{-1}(x,t)\\
\label{asparam}
&=I+\frac{\sqrt{s}B^{R}(x,t)}{\sqrt{2(4st)^{\alpha/(2-\alpha)}}(z+s)}
+O\left(t^{\frac{\alpha}{\alpha-2}}\right),
\end{align}
where 
\begin{equation}
B^R(x,t)=\begin{pmatrix}
0 & -i\beta^R(s,t)e^{isx^{\alpha} + 2\hat\chi(-s,s,t)}
\left(\frac{2x^{\alpha}}{s}\right)^{-i\hat\nu(-s,t)} \\
i\gamma^R(s,t)e^{-isx^{\alpha} - 2\hat\chi(-s,s,t)}
\left(\frac{2x^{\alpha}}{s}\right)^{i\hat\nu(-s,t)} & 0
\end{pmatrix}.
\end{equation}
Estimations of the jump matrix $(\check J^{R}-I)$ and the Cauchy operator
similar to those  in  \cite{RS} (see also \cite{Len15})
imply  that the main term in the large-$t$ development 
of $\check M^{R}$ in (\ref{M-int-rep}) is given by the integral along the circle
$|\zeta+s|=\varepsilon$, which in turn gives
\begin{equation}\label{M^R}
\check M^R(x,t,z)=I-\frac{1}{2\pi i}\int_{|\zeta+s|=\varepsilon}
\frac{\tilde{B}^{R}(\alpha,s,t)}{(\zeta+s)(\zeta-z)}
\,d\zeta+O\left(t^{\frac{\alpha}{\alpha-2}}\ln t\right),
\quad |\zeta+s|>\varepsilon,
\end{equation}
where
\begin{subequations}\label{tilde-Bz}
\begin{align}
&\tilde{B}^R_{11}(\alpha,s,t)=\tilde{B}^R_{22}(\alpha,s,t)=0,\\
&\tilde{B}^R_{12}(\alpha,s,t)=\tilde\beta^{R}(\alpha,s,t)
\exp\left\{is(4st)^{\frac{\alpha}{2-\alpha}}
-i\alpha\frac{\hat\nu(-s)}{2-\alpha}\ln 4st\right\}t^{\frac{\alpha}{2\alpha-4}},\\
&\tilde{B}^R_{21}(\alpha,s,t)=\tilde\gamma^{R}(\alpha,s,t)
\exp\left\{-is(4st)^{\frac{\alpha}{2-\alpha}}
+i\alpha\frac{\hat\nu(-s)}{2-\alpha}\ln 4st\right\}t^{\frac{\alpha}{2\alpha-4}},
\end{align}
\end{subequations}
with
\begin{subequations}\label{be-ga-Rz-tilde}
\begin{align}
&\tilde\beta^{R}(\alpha,s,t)=i\beta^{R}(s,t)\exp\left\{i\hat\nu(-s,t)
\ln\frac{s}{2}
+\frac{1-\alpha}{2-\alpha}\ln s+\frac{\alpha+2}{2\alpha-4}\ln 2+2\hat\chi(-s,s,t)
\right\},\\
&\tilde\gamma^{R}(\alpha, s,t)=-i\gamma^{R}(s,t)\exp\left\{-i\hat\nu(-s,t)
\ln\frac{s}{2}
+\frac{1-\alpha}{2-\alpha}\ln s+\frac{\alpha+2}{2\alpha-4}\ln 2-2\hat\chi(-s,s,t)
\right\}.
\end{align}
\end{subequations}
Particularly, since $\check M^{R}=\hat M^{R}$ for all $z$ with
$|z+s|>\varepsilon$, we have 
\begin{equation}\label{kinftyz}
\lim\limits_{z\to\infty}z\left(\hat M^{R}(x,t,z)-I\right)=
\tilde{B}^{R}(\alpha,s,t)+O\left(t^{\frac{\alpha}{\alpha-2}}\ln t\right),
\end{equation}
as well as
\begin{subequations}\label{MR0z}
\begin{align}
\hat M^R(x,t,0)=I+\frac{\tilde B^{R}(\alpha,s,t)}{s}
+O\left(t^{\frac{\alpha}{\alpha-2}}\ln t\right),\\
\label{MRik_1z}
\hat M^R(x,t,ik_1x^{1-\alpha})=I+\frac{\tilde B^{R}(\alpha,s,t)}
{s+ik_1x^{1-\alpha}}+O\left(t^{\frac{\alpha}{\alpha-2}}\ln t\right).
\end{align}
\end{subequations}

Now observe that  $P_{12}(x,t)$ and $P_{21}(x,t)$  (see (\ref{sol+0z})) can be evaluated similarly to \cite{RS2}, with $k_1x^{1-\alpha}$, $\hat c_0(x,t)$ and $s$ replacing  $k_1$, $c_0(\xi)$ and $\xi$ in \cite{RS2} respectively (we drop the arguments of $\tilde{B}^{R}$ and use the standard notation for its entries):
\begin{align}
\nonumber
g_1(x,t)&=ik_1x^{1-\alpha}+O\left(t^{\frac{\alpha}{\alpha-2}}\ln t\right),
&&g_2(x,t)=\frac{ik_1x^{1-\alpha}}{s+ik_1x^{1-\alpha}}\tilde{B}^{R}_{21}
+O\left(t^{\frac{\alpha}{\alpha-2}}\ln t\right),\\
\nonumber
h_1(x,t)&=\hat c_0(x,t)+\frac{ik_1}{s}x^{1-\alpha}\tilde{B}^{R}_{12}
+O\left(t^{\frac{\alpha}{\alpha-2}}\ln t\right),
&&h_2(x,t)=ik_1x^{1-\alpha}+\frac{\hat c_0(x,t)}{s}\tilde{B}^{R}_{21}
+O\left(t^{\frac{\alpha}{\alpha-2}}\ln t\right).
\end{align}
Consequently,  we have (dropping again the arguments)
\begin{subequations}\label{gh2}
\begin{align}
g_1h_1&=ik_1x^{1-\alpha}\hat c_0-\frac{k_1^2x^{2-2\alpha}}{s}\tilde{B}^{R}_{12}
+\tilde R(\alpha,t),
&&g_1h_2=-k_1^2x^{2-2\alpha}+\frac{ik_1x^{1-\alpha}\hat c_0}
{s}\tilde{B}^{R}_{21}
+\tilde R(\alpha,t),\\
g_2h_1&=\frac{ik_1x^{1-\alpha}\hat c_0}{s+ik_1x^{1-\alpha}}\tilde{B}^{R}_{21}
+O\left(t^{\frac{2\alpha-1}{\alpha-2}}\tilde\beta^R\tilde\gamma^R\right),
&&g_2h_2=-\frac{k_1^2x^{2-2\alpha}}{s+ik_1x^{1-\alpha}}\tilde{B}^{R}_{21}
+\tilde R(\alpha,t).
\end{align}
\end{subequations}
where $\tilde R(\alpha,t)=O\left(t^{\frac{2\alpha-1}{\alpha-2}}\ln t\right)$. Substituting (\ref{gh2}) into (\ref{P}), and making a rough estimation 
$O\left(t^{\frac{2\alpha-1}{\alpha-2}}\tilde\beta^R\tilde\gamma^R\right)=
O\left(t^{\frac{2\alpha-1}{\alpha-2}}\ln t\right)$
(see (\ref{tilde-be-ga-R})) we obtain

\begin{subequations}\label{P-asz}
\begin{align}
P_{12}(x,t)&=-\frac{i\hat c_0(x,t)}{k_1x^{1-\alpha}}
+\frac{\tilde{B}^R_{12}(\alpha,s,t)}{s}
+\frac{i\hat c_0^2(x,t)}{s k_1x^{1-\alpha}(s+ik_1x^{1-\alpha})}
\tilde{B}^R_{21}(\alpha,s,t)+O\left(t^{\frac{\alpha}{\alpha-2}}\ln t\right),\\
P_{21}(x,t)&=-\frac{\tilde{B}^R_{21}(\alpha,s,t)}{s+ik_1x^{1-\alpha}}
+O\left(t^{\frac{1}{\alpha-2}}\ln t\right).
\end{align}
\end{subequations}
Collecting (\ref{sol+0z}), (\ref{kinftyz}), (\ref{P-asz}) and taking into account that $x^{\alpha-1}\tilde{B}^R=O\left(t^{-\frac{1}{2}}\sqrt{\ln t}\right)$ and $\hat c_0(x,t)=\frac{A}{2i}x^{1-\alpha}\hat\delta^2(0,s,t)$ we get
\begin{subequations}\label{gen-as}
\begin{align}
&q(x,t)=A\hat\delta^2(0,s,t)+\frac{A^2}{2k_1s}\hat\delta^4(0,s,t)\tilde B_{21}^{R}(\alpha,s,t)-\frac{2k_1}{s}\tilde B_{12}^{R}(\alpha,s,t)
+\hat R_1(\alpha,t), &&x>0,
\\
&q(-x,t)=\frac{2s}{k_1}x^{2\alpha-2}\overline{\tilde{B}_{21}^{R}(\alpha,s,t)}
+\hat R_2(\alpha,t), && x>0,
\end{align}
\end{subequations}
where (notice that roughly $\tilde\gamma^{R}=O(\sqrt{\ln t})$ for both Case I and Case II, see (\ref{ga-Rz-tilde}))
\begin{equation}
\hat R_1(\alpha,t)=
\begin{cases}
O\left(t^{-\frac{1}{2}}\sqrt{\ln t}\right),
&\alpha\in\left(\frac{2}{3},1\right),\\
O\left(t^{\frac{\alpha}{\alpha-2}}\ln t\right),
&\alpha\in\left(0,\frac{2}{3}\right].
\end{cases}
\hat R_2(\alpha,t)=
\begin{cases}
O\left(t^{\frac{6-5\alpha}{2\alpha-4}}\sqrt{\ln t}\right),
&\alpha\in\left(\frac{4}{5},1\right),\\
O\left(t^{\frac{1}{\alpha-2}}\ln t\right),
&\alpha\in\left(0,\frac{4}{5}\right].
\end{cases}
\end{equation}
Finally, taking into account (\ref{tilde-Bz}), (\ref{gen-as}), the asymptotics of $\hat\nu$ and $\hat\chi$ given by Proposition \ref{aschinu} and the asymptotics of $\tilde\beta^R$ and $\tilde\gamma^R$ given by  Proposition \ref{begaas}, the statements of Theorem \ref{th1} follow.
\end{proof}

\section{Matching with the asymptotics along straight lines}
\label{merge}
In this section we  compare the asymptotics presented in 
 Theorem \ref{th1} with those obtained in \cite{RS2} for the case of 
constant $\xi=\frac{x}{4t}$. To do this, we consider the asymptotics along the curves $t=\frac{x^{2-\alpha}}{4s}$, $s>0$ obtained in Theorem \ref{th1}, and take the limit $\alpha\to1$ (see Figure \ref{regions}). First, we consider asymptotics of $q(x,t)$, $x>0$, which takes the form:
\begin{equation}\label{a1as0}
q(x,t)\sim 
\begin{cases}
\mathrm{Q}e^{i\Psi_I(1, s, t)}, & \mbox{Case I},\\
\mathrm{Q}e^{i\Psi_{II}(1, s, t)}, & \mbox{Case II}.
\end{cases}
\end{equation}
Taking into account that $\delta(0,\xi)$ (see (\ref{xidel})) has the following form as $\xi\to +0$:
\begin{equation}\label{asdelta}
\delta(0,\xi)\sim
\begin{cases}
\exp\left\{
\frac{i}{\pi}\ln\xi\cdot\ln\frac{A2|a_2(0)|}{2\xi}+\hat\chi_0(\xi)\right\},
& \mbox{Case I},
\\
\exp\left\{\frac{i}{2\pi}\ln\xi\cdot\ln a_{11} a_{21}+\hat\chi_1\right\},
&\mbox{Case II},
\end{cases}
\end{equation}
we conclude that (\ref{a1as0}) are the same as the main terms in \cite{RS2} as $\xi\to+0$, with $s$ replacing $\xi$.

Now let us consider the asymptotics of $q(-x,t)$ for $x>0$ and $\alpha\to 1$:
\begin{equation}\label{a1as1}
q(-x,t)\sim 
\begin{cases}
t^{-\frac{1}{2}}\sqrt{\frac{1-\alpha}{\pi(2-\alpha)}\ln t}\cdot
\frac{C_1(s)\exp\{4is^2t\}}
{\exp\left\{-\frac{i}{\pi}\ln\frac{A|a_2(0)|}{2s}
\ln\frac{1-\alpha}{\pi(2-\alpha)}\right\}}, & \mbox{Case I},\\
t^{-\frac{1}{2}}\hat\alpha_1\exp\{4its^2-i\nu(0)\ln t\}, & \mbox{Case II},
\end{cases}
\end{equation}
where
$
C_1(s)=\frac{4i}{A}s\exp\{\frac{i}{\pi}\ln\frac{A|a_2(0)|}{2s}\ln\frac{s}{2}
-\frac{3}{2}\ln 2-2\overline{\hat\chi_{-s}(s)}\}
$
and
\begin{equation}\label{hatal}
\hat\alpha_1=-\frac{\sqrt{\pi}\exp\{-\frac{\pi}{2}\nu(0)+\frac{\pi i}{4}-2\overline{\hat\chi_1}-3i\nu(0)\ln 2\}}
{\frac{b(0)}{sa_{21}}\Gamma(-i\nu(0))},
\end{equation}
with
$
\nu(0)=\frac{1}{2\pi}\ln a_{11}a_{21}.
$
Assuming that $(1-\alpha)\ln t=O(1)$ as $\alpha\to 1$ and $t\to\infty$, the asymptotics (\ref{a1as1}) in Case I has the form $|q(x,t)|=C(\alpha,s)t^{-1/2}$, which is consistent with that obtained in \cite{RS2}. Then, straightforward calculations show that (\ref{a1as1}) in  Case II is precisely the same as in \cite{RS2}, with $s$ replacing $\xi$, since $\hat\alpha_1$ and $\nu(0)$ are consistent with respectively $\alpha_1(\xi)$ and $\nu(-\xi)$ defined in \cite{RS2} in  Case II with $\xi=0$  (we have the minus sign at  the r.h.s. of (\ref{hatal}) since here we do not change $x$ to $-x$ for $x<0$).


\bigskip

\noindent\textbf{Acknowledgments.}
D.S. gratefully acknowledges the hospitality of the University of Vienna
and the partial support of the Austrian Science Found (Grant FWF no. P31651). Ya.R. thankfully acknowledges the partial support from the Akhiezer Foundation.

\section*{Appendix A}
\label{appA}
\textit{Proof of  Proposition \ref{aschinu}}.

The behavior of $(1+r_1(k)r_2(k))$ as $k\to0$:
\begin{equation}\label{r1r20}
1+r_1(k)r_2(k)=
\begin{cases}
\frac{4k^2}{A^2a_2^2(0)} + O(k^3), &\mbox{Case I},\\
\frac{1}{a_{11}a_{21}} + O(k), &\mbox{Case II},
\end{cases}
\end{equation}
implies, by straightforward calculations,  that $\hat\nu(-s,t)$ has asymptotic behavior described by (\ref{nuas}) (recall that $a_{11}a_{21}>0$ in  Case II due to the Assumption 2).

Taking into account the definition of $\hat r_j(z)$ (see (\ref{r12z})) and factoring out the term $x^{1-\alpha}$ in the logarithm, the function $\hat\chi(z,s,t)$ defined by (\ref{chi}) can be written as follows (recall that $\alpha\in(0,1)$ and $s>0$):
\begin{align}
\nonumber
&\hat\chi(z,s,t)=-\frac{1}{2\pi i}
\int_{-\infty}^{-sx^{\alpha-1}}\ln(z-x^{1-\alpha}\zeta)\,
d_{\zeta}\ln(1+r_1(\zeta)r_2(\zeta))\\
\label{chising1}
&=\frac{\alpha-1}{2\pi i}\ln x\cdot\ln(1+r_1(-sx^{\alpha-1})r_2(-sx^{\alpha-1}))
-\frac{1}{2\pi i}\int_{-\infty}^{-sx^{\alpha-1}}\ln(zx^{\alpha-1}-\zeta)\,
d_{\zeta}\ln(1+r_1(\zeta)r_2(\zeta)).
\end{align}
The real part of the first term in (\ref{chising1}) can be estimated as
 $O\left(t^{\frac{1-\alpha}{\alpha-2}}\ln t\right)$, $t\to\infty$ (see (\ref{r1r20}) and  (\ref{arg0})), whereas integrating by parts we conclude that the real part of the second term is the r.h.s. of (\ref{rechi}).

Now let us evaluate $\hat\chi(z,s,t)$ at $z=0$ and $z=-s$ in  Cases I and II. In  Case II, the function $\ln(1+r_1(k)r_2(k))$ is bounded as $k\to 0$, so for all $z\geq-s$ we have the asymptotics (\ref{chiII}).

In  Case I, the function $\ln(1+r_1(k)r_2(k))$ has a singularity at $k=0$
and thus a neighborhood of $\zeta=-sx^{\alpha-1}$ in (\ref{chising1}) 
is to be treated separately. For all $z\geq-s$ we have:
\begin{align}
\nonumber
&\int_{-1+zx^{\alpha-1}}^{-sx^{\alpha-1}}
\ln(zx^{\alpha-1}-\zeta)\,d_{\zeta}\ln(1+r_1(\zeta)r_2(\zeta))
\\&=
\nonumber
\int_{-1+zx^{\alpha-1}}^{-sx^{\alpha-1}}
\ln(zx^{\alpha-1}-\zeta)\,d_{\zeta}\ln\frac{1+r_1(\zeta)r_2(\zeta)}{\zeta^2}
+2\int_{-1+zx^{\alpha-1}}^{-sx^{\alpha-1}}
\ln(zx^{\alpha-1}-\zeta)\,d_{\zeta}\ln(-\zeta)
\\&=
\label{neigh0chi}
\int_{-1}^{0}
\ln(-\zeta)\,d_{\zeta}\ln\frac{1+r_1(\zeta)r_2(\zeta)}{\zeta^2}
+2\int_{-1+zx^{\alpha-1}}^{-sx^{\alpha-1}}
\ln(zx^{\alpha-1}-\zeta)\,d_{\zeta}\ln(-\zeta) + O\left(t^{\frac{1-\alpha}{\alpha-2}}\ln t\right).
\end{align}
Collecting (\ref{chising1}) and (\ref{neigh0chi}) with $z=0$ we arrive at (\ref{chiIa}).

For $z=-s$, the last term in (\ref{neigh0chi}) has the form (here we use the notation 
$y:=-sx^{\alpha-1}$):
\begin{align}
\nonumber
2\int_{-1+y}^{y}\ln(y-\zeta)\,\frac{d\zeta}{\zeta}&=\ln^2 y
+2\int_{-1+y}^{y}\frac{\ln\left(1-\frac{y}{\zeta}\right)}{\zeta}\,d\zeta+O(y),
=\ln^2 y-2\int_{-1+y}^{y}\frac{\sum_{n=1}^{\infty}\frac{y^n}{n\zeta^n}}{\zeta}+O(y)
\\
\label{neigh0chi1}
&=\ln^2 y+2\sum_{n=1}^{\infty}\frac{1}{n^2}+O(y)
\quad y\to-0.
\end{align}
Observing that $\sum_{n=1}^{\infty}\frac{1}{n^2}=\frac{\pi^2}{6}$ and combining (\ref{chising1}), (\ref{neigh0chi}) with $z=-s $, and (\ref{neigh0chi1}), we obtain (\ref{chiIb}).

\section*{Appendix B}\label{appB}
\begin{proposition}
The long-time asymptotics of $\beta^{R}(s,t)$ and $\gamma^{R}(s,t)$ (see (\ref{be-ga-Rz})) have the form:
\begin{align}
&\beta^{R}(s,t)=\begin{cases}
\frac{A\sqrt{\ln t}\exp
\{i\phi_2(\alpha)\ln 4st\cdot\ln\ln 4st+i\tilde\phi_3(\alpha)\ln 4st
+i\phi_4(s)\ln\ln 4st\}}
{2k_1\exp\left\{\left(\frac{i}{\pi}\ln\frac{2s}{A|a_2(0)|}-\frac{1}{2}\right)	
	\ln\frac{1-\alpha}{\pi(2-\alpha)}\right\}
}+O\left(\frac{1}{\sqrt{\ln t}}\right),&\mbox{Case I},\\
\frac{\sqrt{2\pi}a_{11}^{3/4}a_{21}^{-1/4}e^{-\frac{\pi i}{4}}}
{k_1b(0)\Gamma\left(-\frac{i}{2\pi}\ln a_{11}a_{21}\right)}
+O\left(t^{\frac{\alpha-1}{2-\alpha}}\right), &\mbox{Case II},
\end{cases}
\\
&\gamma^{R}(s,t)=\begin{cases}
\frac{2k_1\sqrt{\ln t}\exp
\{-i\phi_2(\alpha)\ln 4st\cdot\ln\ln 4st-i\tilde\phi_3(\alpha)\ln 4st
-i\phi_4(s)\ln\ln 4st\}}
{A\exp\left\{\left(\frac{i}{\pi}\ln\frac{A|a_2(0)|}{2s}-\frac{1}{2}\right)	
	\ln\frac{1-\alpha}{\pi(2-\alpha)}\right\}
}+O\left(\frac{1}{\sqrt{\ln t}}\right),&\mbox{Case I},\\
\frac{\sqrt{2\pi}k_1a_{11}^{-1/4}a_{21}^{3/4}e^{-\frac{3\pi i}{4}}}
{\overline{b(0)}\Gamma\left(\frac{i}{2\pi}\ln a_{11}a_{21}\right)}
+O\left(t^{\frac{\alpha-1}{2-\alpha}}\right), &\mbox{Case II},
\end{cases}
\end{align}
where $\phi_2$ and $\phi_4$ are given by (\ref{phij}), and 
\begin{equation}\label{tildephi3}
\tilde\phi_3(\alpha)=\frac{1-\alpha}{\pi(2-\alpha)}
\left(\ln\frac{1-\alpha}{\pi(2-\alpha)}-1\right),\\
\end{equation}
\end{proposition}
\begin{proof}
First, observe that since
\begin{equation}\nonumber
r_1(k)=
\begin{cases}
\frac{2k}{iA} + O(k^2),\,k\to 0,&\mbox{Case I},\\
\frac{b(0)k}{a_{11}} + O(k^2),\,k\to 0,&\mbox{Case II},
\end{cases}
\,\,
r_2(k)=
\begin{cases}
\frac{A}{2ik} + O(1),\,k\to 0,&\mbox{Case I},\\
\frac{\overline{b(0)}}{a_{21}k} + O(1),\,k\to 0,&\mbox{Case II},
\end{cases}
\end{equation}
the large-$t$ behavior of $\hat r^R_j(-s,t)$, $j=1,2$ (see (\ref{r12z}) and (\ref{r12R})) has the form
\begin{equation}\label{rjas}
\hat r^R_1(-s)=
\begin{cases}
\frac{-2k_1}{A} + O\left(t^{\frac{\alpha-1}{2-\alpha}}\right),
&\mbox{Case I},\\
\frac{-ik_1b(0)}{a_{11}} + O\left(t^{\frac{\alpha-1}{2-\alpha}}\right),
&\mbox{Case II},
\end{cases}
\hat r^R_2(-s,t)=
\begin{cases}
\frac{A}{2k_1} + O\left(t^{\frac{\alpha-1}{2-\alpha}}\right),
&\mbox{Case I},\\
\frac{i\overline{b(0)}}{k_1a_{21}} + O\left(t^{\frac{\alpha-1}{2-\alpha}}\right),
,&\mbox{Case II}.
\end{cases}
\end{equation}
Second, the asymptotics of $e^{-\frac{\pi}{2}\hat\nu(-s,t)}$ has the form (see (\ref{nuas}))
\begin{equation}\label{enuas}
e^{-\frac{\pi}{2}\hat\nu(-s,t)}=
\begin{cases}
t^{\frac{\alpha-1}{4-2\alpha}}
\exp\left\{\frac{\alpha-1}{4-2\alpha}\ln 4s-\frac{1}{2}\ln\frac{A|a_2(0)|}{2s}\right\}
 + O\left(t^{\frac{3\alpha-3}{4-2\alpha}}\right),
\quad&\mbox{ Case I},\\
\left(a_{11}a_{21}\right)^{-1/4} + O\left(t^{\frac{\alpha-1}{2-\alpha}}\right),
\quad&\mbox{ Case II}.
\end{cases}
\end{equation}
Further, taking into account the asymptotic expansion of the Euler's Gamma function (see e.g. \cite{OLBC})
\begin{equation}
\Gamma(az+b) = \sqrt{2\pi}e^{-az}(az)^{az+b-1/2}\left(1+O(z^{-1})\right),\quad a>0,\quad b\in\mathbb{C},\quad |\arg z|<\pi-\delta,
\end{equation}
we conclude that
\begin{equation}\label{gammaas}
\Gamma(\pm i\hat\nu(-s))=
\begin{cases}
\frac{t^{\frac{\alpha-1}{4-2\alpha}}}{\sqrt{\ln t}}g^{as}_{\pm}(\alpha, s)
e^{\pm i\phi_2(\alpha)\ln 4st\cdot\ln\ln 4st\pm i\tilde\phi_3(\alpha)\ln 4st
\pm i\phi_4(s)\ln\ln 4st}
+O\left(\frac{t^{\frac{\alpha-1}{4-2\alpha}}}{\ln^{3/2} t}\right),
&\mbox{Case I},\\
\Gamma\left(\pm\frac{i}{2\pi}\ln a_{11}a_{21}\right)
+O\left(t^{\frac{\alpha-1}{2-\alpha}}\right),&\mbox{Case II},
\end{cases}
\end{equation}
where
$$
g^{as}_{\pm}(\alpha, s)=\sqrt{2\pi}\exp\left\{
\frac{\alpha-1}{4-2\alpha}\ln 4s+\left(\pm\frac{i}{\pi}\ln\frac{A|a_2(0)|}{2s}
-\frac{1}{2}\right)
\left(\ln\frac{1-\alpha}{\pi(2-\alpha)}\pm i\frac{\pi}{2}
\right)
\right\},
$$
$\phi_2$ and $\phi_4$ are given by (\ref{phij}), and $\tilde\phi_3$ is given by (\ref{tildephi3}).
Finally, collecting (\ref{be-ga-Rz}), (\ref{rjas}), (\ref{enuas}) and (\ref{gammaas}) we arrive at the result.
\end{proof}
Taking into account the large time behavior of $\hat\nu(-s,t)$ and $\hat\chi(-s,s,t)$ (see Proposition \ref{aschinu}) for the both Cases I and II, we arrive at the following 
\begin{proposition}\label{begaas}
The long-time asymptotics of $\tilde\beta^{R}(s)$ and $\tilde\gamma^{R}(s)$ 
(see (\ref{be-ga-Rz-tilde})) has the form:
\begin{subequations}\label{tilde-be-ga-R}
\begin{align}
&\tilde\beta^{R}(\alpha,s,t)=\begin{cases}
\tilde\beta^{R}_{I,\,as}(\alpha,s)e^{i\tilde\Psi_+(\alpha,s,t)}
\sqrt{\ln t}+O\left(\frac{1}{\sqrt{\ln t}}\right),&\mbox{Case I},\\
\tilde\beta^{R}_{II,\,as}(\alpha,s)\exp\{i\hat\phi_5(\alpha)\ln 4st\}
+O\left(t^{\frac{\alpha-1}{2-\alpha}}\right), &\mbox{Case II},
\end{cases}
\\
\label{ga-Rz-tilde}
&\tilde\gamma^{R}(\alpha,s,t)=\begin{cases}
\tilde\gamma^{R}_{I,\,as}(\alpha, s)e^{i\tilde\Psi_-(\alpha,s,t)}
\sqrt{\ln t}+O\left(\frac{1}{\sqrt{\ln t}}\right),&\mbox{Case I},\\
\tilde\gamma^{R}_{II,\,as}(\alpha,s)\exp\{-i\hat\phi_5(\alpha)\ln 4st\}
+O\left(t^{\frac{\alpha-1}{2-\alpha}}\right), &\mbox{Case II},
\end{cases}
\end{align}
\end{subequations}
where $\hat \phi_5(\alpha)=\frac{(1-\alpha)}{2\pi(\alpha-2)}\ln(a_{11}a_{21})$ and
\begin{equation}
\tilde\Psi_{\pm}(\alpha,s,t)=\pm \hat\phi_1(\alpha)\ln^2 4st
\pm \phi_2(\alpha)\ln 4st\cdot\ln\ln 4st
\pm \hat\phi_3(\alpha,s)\ln 4st\pm \phi_4(s)\ln\ln 4st,
\end{equation}
with $\phi_2$, and $\phi_4$ are given by (\ref{phij}) and
\begin{equation}\label{hatphi3}
\hat\phi_1(\alpha)=\frac{(1-\alpha)^2}{\pi(2-\alpha)^2},\quad
\hat\phi_3(\alpha,s)=\frac{1-\alpha}{\pi(2-\alpha)}
\left(
\ln\frac{1-\alpha}{\pi(2-\alpha)}+\ln\frac{2s}{A^2a_2^2(0)}-1
\right),
\end{equation}
and the constants have the form
\begin{subequations}\label{tilde-be-ga-as}
\begin{align}
&\tilde\beta^{R}_{I,\,as}(\alpha,s)=\frac{
iA\exp\left\{\frac{i}{\pi}\ln\frac{A|a_2(0)|}{2s}
\ln\frac{s}{2}
+\frac{1-\alpha}{2-\alpha}\ln s+\frac{\alpha+2}{2\alpha-4}\ln 2+2i\hat\chi_{-s}(s)
\right\}
}{
2k_1\exp\left\{\left(\frac{i}{\pi}\ln\frac{2s}{A|a_2(0)|}-\frac{1}{2}\right)	
\ln\frac{1-\alpha}{\pi(2-\alpha)}\right\}
},
\\
&\tilde\beta^{R}_{II,\,as}(\alpha,s)=
\frac{\sqrt{2\pi}a_{11}^{3/4}a_{21}^{-1/4}e^{\frac{\pi i}{4}}
\exp\left\{\frac{i}{2\pi}\ln (a_{11}a_{21})
\ln\frac{s}{2}
+\frac{1-\alpha}{2-\alpha}\ln s+\frac{\alpha+2}{2\alpha-4}\ln 2+2i\hat\chi_1
\right\}
}
{k_1b(0)\Gamma\left(-\frac{i}{2\pi}\ln a_{11}a_{21}\right)},
\\
&\tilde\gamma^{R}_{I,\,as}(\alpha,s)=\frac{
-2ik_1\exp\left\{-\frac{i}{\pi}\ln\frac{A|a_2(0)|}{2s}
\ln\frac{s}{2}
+\frac{1-\alpha}{2-\alpha}\ln s+\frac{\alpha+2}{2\alpha-4}\ln 2-2i\hat\chi_{-s}(s)
\right\}
}{
A\exp\left\{\left(\frac{i}{\pi}\ln\frac{A|a_2(0)|}{2s}-\frac{1}{2}\right)	
\ln\frac{1-\alpha}{\pi(2-\alpha)}\right\}
},
\\
&\tilde\gamma^{R}_{II,\,as}(\alpha,s)=
\frac{\sqrt{2\pi}k_1a_{11}^{-1/4}a_{21}^{3/4}e^{\frac{3\pi i}{4}}
\exp\left\{-\frac{i}{2\pi}\ln (a_{11}a_{21})\ln\frac{s}{2}
+\frac{1-\alpha}{2-\alpha}\ln s+\frac{\alpha+2}{2\alpha-4}\ln 2-2i\hat\chi_1
\right\}
}
{\overline{b(0)}\Gamma\left(\frac{i}{2\pi}\ln a_{11}a_{21}\right)}.
\end{align}
\end{subequations}
\end{proposition}

\end{document}